\documentclass[11pt,oneside,english,reqno]{amsart}
\usepackage[T1]{fontenc}
\usepackage[utf8]{inputenc}
\usepackage[a4paper]{geometry}
\usepackage{mathrsfs}
\usepackage{bm}
\usepackage{amstext}
\usepackage{amsthm}
\usepackage{amssymb}
\usepackage{amsmath}
\usepackage{bbm}
\usepackage{shuffle}
\usepackage{color}
\usepackage{float}
\usepackage{enumerate}
\usepackage{verbatim} 
\usepackage{stmaryrd}
\usepackage{enumitem}
\usepackage{graphicx}
\usepackage{todonotes}
\usepackage{cancel}
\usepackage[normalem]{ulem}
\usepackage{subcaption}
\usepackage{placeins}
\usepackage{caption}
\usepackage{subcaption}
\usepackage{float}
\usepackage{booktabs}   
\usepackage{multirow}   

\usepackage{hyperref}

\usetikzlibrary{arrows.meta, shapes.geometric}
\usepackage{tikz, pgfplots}
\usepackage{tikz-cd}
\usetikzlibrary{matrix,chains,positioning,decorations.pathreplacing,arrows}
\usetikzlibrary{calc,positioning}

\usepackage{tikz}
\usetikzlibrary{arrows.meta,positioning}
\makeatletter
\numberwithin{equation}{section}
\numberwithin{figure}{section}
\theoremstyle{plain}
\newtheorem{thm}{\protect\theoremname}
\numberwithin{thm}{section}
\theoremstyle{definition}
\newtheorem{defn}[thm]{\protect\definitionname}
\theoremstyle{remark}
\newtheorem{rem}[thm]{\protect\remarkname}
\theoremstyle{plain}
\newtheorem{lem}[thm]{\protect\lemmaname}
\theoremstyle{plain}
\newtheorem{prop}[thm]{\protect\propositionname}
\theoremstyle{plain}

\theoremstyle{plain}
\newtheorem{example}[thm]{\protect\examplename}
\DeclareMathOperator{\argmin}{argmin}

\makeatother

\usepackage{babel}
\providecommand{\corollaryname}{Corollary}
\providecommand{\definitionname}{Definition}
\providecommand{\lemmaname}{Lemma}
\providecommand{\propositionname}{Proposition}
\providecommand{\remarkname}{Remark}
\providecommand{\theoremname}{Theorem}
\providecommand{\examplename}{Example}

\newcommand{\cA}{\mathcal{A}}
\newcommand{\cB}{\mathcal{B}}

\newcommand{\cF}{\mathcal{F}}

\newcommand{\cH}{\mathcal{H}}
\newcommand{\cI}{\mathcal{I}}

\newcommand{\cK}{\mathcal{K}}
\newcommand{\cL}{\mathcal{L}}

\newcommand{\cN}{\mathcal{N}}

\newcommand{\cP}{\mathcal{P}}

\newcommand{\cR}{\mathcal{R}}
\newcommand{\cS}{\mathcal{S}}

\newcommand{\cV}{\mathcal{V}}

\newcommand{\cX}{\mathcal{X}}

\newcommand{\sig}{\mathbf{Sig}}
\newcommand{\bsig}{\mathbf{BSig}}

\newcommand{\trees}{\mathcal{T}}

\newcommand{\KK}{\mathbb{K}}

\newcommand{\NN}{\mathbb{N}}

\newcommand{\RR}{\mathbb{R}}

\newcommand{\bbX}{\Bar{\mathbb{X}}}

\newcommand{\bB}{\mathbf{B}}

\newcommand{\bK}{\mathbf{K}}

\newcommand{\bW}{\mathbf{W}}
\newcommand{\bX}{\mathbf{X}}
\newcommand{\bx}{\mathbf{x}}
\newcommand{\bY}{\mathbf{Y}}
\newcommand{\by}{\mathbf{y}}

\newcommand{\se}{\mathsf e}

\renewcommand{\d}{\mathrm{d}}
\newcommand{\bu}{\mathbf{u}}
\newcommand{\bv}{\mathbf{v}}

\newcommand{\bw}{\mathbf{w}}

\newcommand{\hopf}{\mathscr{H}}

\newcommand{\bg}{\bm{g}}
\newcommand{\bff}{\bm{f}}

\newcommand{\dd}{\mathop{}\!\mathrm{d}}

\newenvironment{tree}
               {\begin{tikzpicture}[every node/.style={circle,draw,fill,minimum size=3pt,inner sep=0pt,outer sep=0pt},line cap=round,baseline = .1pt]}
               {\end{tikzpicture}}

\newcommand{\innerprod}[1]{\left\langle #1 \right\rangle}

\newcommand{\R}{\mathbb{R}}

\usepackage[external]{forest}

\def\dtR<#1>{\Forest{[#1]}}
\def\dtI<#1#2>{\Forest{[#1[#2]]}}
\def\dtII<#1#2#3>{\Forest{[#1[#2[#3]]]}}
\def\dtV<#1#2#3>{\Forest{[#1[#2][#3]]}}

\forestset{
  decor/.style = {
    label/.expanded = {[inner sep = 0.1ex, font=\unexpanded{\tiny}]right:{$#1$}}
  },
  root/.style = {minimum size = 0.5ex},
  decorated/.style = {
    for tree = {
      circle, fill, inner sep = 0ex, minimum size = 0.5ex,
      grow' = north, l = 0, l sep = 0.8ex, s sep = 0.5em,
      fit = tight, parent anchor = center, child anchor = center,
      delay = {decor/.option = content, content =}
    }
  },
  default preamble = {decorated, root},
  begin draw/.code={ \begin{tikzpicture}[baseline={([yshift=-0.5ex]current bounding box.center)}] },
}

\title{Branched Signature Kernel Solvers for ODEs with  rough Single-Trajectory signals
}

\author{
Munawar Ali$^*$\qquad Qi Feng$^\dagger$\qquad Charlie Pyle$^\sharp$\qquad George Xu$^{\sharp\sharp}$
}
\thanks{$^*$: Department of Mathematics, Florida State University,
Tallahassee, FL 32306;
e-mail: \texttt{ma22bm@fsu.edu}.}
\thanks{$^\dagger$: Department of Mathematics, Florida State University,
Tallahassee, FL 32306;
e-mail: \texttt{qfeng2@fsu.edu}.
This author is partially supported by the National Science Foundation
under grant \#DMS-2420029.}
\thanks{$^\sharp$: Department of Mathematics, Texas A\&M University,
College Station, TX 77843;
e-mail: \texttt{charliepyle@tamu.edu}.
This author is partially supported by the National Science Foundation
under grant \#DMS-2420029 through the Research Experience for
Undergraduates (REU) program at Florida State University.}
\thanks{$^{\sharp\sharp}$: Department of Mathematics, Rutgers University,
Piscataway, NJ 08854;
e-mail: \texttt{gtx1@scarletmail.rutgers.edu}.
This author is partially supported by the National Science Foundation
under grant \#DMS-2420029 through the Research Experience for
Undergraduates (REU) program at Florida State University.}

\date{\today}

\begin{document}
\sloppy
\maketitle

\begin{abstract} 
We develop a branched signature kernel solver for linear and nonlinear
ordinary differential equations driven by a \emph{single observed
trajectory} of a possibly rough forcing signal --- 
a setting common within earthquake engineering, finance, biology, and structural health monitoring, where only one forcing realization is available, and the solver must respect the underlying physical law without an ensemble of realizations. We first introduce a count-sampling construction method to turn the single observation into a
hierarchical family of $N+1$ nested training paths on which the branched
signature kernel can be evaluated; this allows the signature kernel
machinery, originally designed for multi-realization regression
problems, to operate on a single-trajectory observation. Then we build a
kernel-collocation framework, which places the ansatz either on the highest-order derivative of the solution or on the solution itself. We prove a universal approximation theorem for the branched signature kernel, leveraging the Hairer--Kelly morphism to express branched signature evaluations through geometric signatures of time-extended paths. The offline solver is extended to a streaming
Test/Train/Retrain protocol with optional closed-form online updates in both linear and nonlinear cases. Numerical experiments on six benchmarks show accurate, stable predictions across all regimes.
\end{abstract}



\section{Introduction}\label{sec:intro}

The numerical solution of ordinary differential equations driven by a \emph{single observed forcing trajectory} arises in many engineering and scientific settings. A structure responding to a recorded ground acceleration during an earthquake, a financial state variable evolving
under a single realized market signal, a biological process excited by a particular environmental input, and a coupled oscillator network perturbed by one realization of noise are all instances of the same problem: solve $\cN\bu(t)=f(t)$ on $[0,T]$ given a single discrete
sample of $f$. Two features distinguish this regime from the classical setting. First, no ensemble of independent realizations of $f$ is available, so ensemble-averaged calibration is impossible. Second, the forcing is typically rough: high-frequency seismic accelerations, fractional Brownian  motion type volatility, and noisy biological signals all exhibit H\"older regularity well below the bounded-variation assumption of classical numerical analysis. A solver that respects the physical law
governing the ODE while learning a representation of the solution from a single rough trajectory is therefore needed.

In the literature, signatures have been successful in encoding path information. Both signatures of geometric rough paths \cite{lyons2007differential,Friz_Victoir_2010} and branched signatures for branched rough paths \cite{gubinelli2010ramification} are fundamental in rough path theory. Uniqueness results of increasing
generality were established in \cite{hambly2010uniqueness,boedihardjo2016signature,cuchiero2023signature}, and universal approximation theorems for continuous functions on path space were proved in
\cite{lyons2014roughpathssignaturesmodelling,
chevyrev2022signaturemomentscharacterizelaws,cuchiero2023signature,chevyrev2025primer,cuchiero2025universal,cox2026universal,ceylan2026universal}. Pairing two signatures via an inner product yields the \emph{signature kernel}, introduced by~\cite{kiraly2019kernels} by lifting static
kernels to path space via the signature transform.
\cite{Salvi_2021_goursat} characterized the signature kernel as the solution of a Goursat PDE, enabling efficient computation, and \cite{toth2024randomfouriersignaturefeatures} developed random Fourier signature features for scalable kernel evaluation. The branched signature model of ~\cite{ali2025branchedsignaturemodel} provides the
rough-path foundation for our branched signature kernel; the key difference is that ~\cite{ali2025branchedsignaturemodel} works at the
feature-vector level, whereas we work at the reproducing-kernel level, which permits both the universal approximation theorem and the single-trajectory collocation solver developed here.

It is often useful to augment paths before taking signatures, to capture additional structure in data streams. Time, basepoint, and lead–lag
augmentations are discussed in \cite{chevyrev2025primer,lyons2025signaturemethodsmachinelearning,chevyrev2022signaturemomentscharacterizelaws,kiraly2019kernels}; time augmentation guarantees signature uniqueness and universal approximation
\cite{kiraly2019kernels,chevyrev2022signaturemomentscharacterizelaws}, while lead–lag augmentation
\cite{gyurko2013extracting,lyons2025signaturemethodsmachinelearning,chevyrev2025primer} captures quadratic variation and is widely used in finance. \cite{DeepSignatureTransforms} introduced learned augmentations via neural networks, leading to the \texttt{Signatory} library. In this paper we view path augmentation as a transformation that lifts a
branched rough path into a geometric rough path; taking the signature kernel of the augmented path then yields the branched signature kernel of the original path. Efficient raw signature implementations are available in
iisignature ~\cite{iisignature}, Signatory ~\cite{kidger2021signatory}, PySigLib ~\cite{shmelev2025pysiglibfastsignaturebased}, and
KerasSig ~\cite{keras_sig}, with Chen's identity
\cite{chevyrev2025primer} enabling reuse of sub-path signatures via ``stream'' features. For signature kernels, KSig ~\cite{ksiglibrary} provides GPU-accelerated dynamic-programming, low-rank, random-Fourier,
and PDE-based implementations, and PySigLib offers PyTorch-compatible kernel computations with backpropagation. Our count-sampling construction is particularly compatible with these libraries: nested
training paths share prefix signatures, so the family
$\{\bff_0,\ldots,\bff_N\}$ can be processed at essentially the cost of the longest path.

Signature-based models have been applied across finance, medicine, and signal processing ~\cite{lemercier2021distributionregressionsequentialdata,lyons2014roughpathssignaturesmodelling,levin2016learningpastpredictingstatistics,zeng2025semiparametricfunctionalclassificationpath,bayer2025localregressionpathspaces,chevyrev2025primer}. In finance, signatures' affinity for rough volatility makes them a natural tool for option pricing, nowcasting, and optimal stopping ~\cite{gyurko2013extracting,mohaddes2025regularizedlearningfractionalbrownian,alos2025volatilitymodelingroughpaths,cohen2023nowcastingsignaturemethods,horvath2023optimalstoppingdistributionregression,lyons2020non}. For differential equations, the closest precedent to our solver is the signature-kernel path-dependent PDE solver of ~\cite{pannier2024pathdependentpdesolverbased}, which fits an RKHS minimum-norm interpolant of the PPDE solution at collocation points but requires \emph{multiple} realizations of the driving path and does not employ branched signatures. Our framework removes both restrictions: a single trajectory suffices via count-sampling, and the branched signature kernel resolves the higher-order brackets that distinguish branched from geometric roughness. Neural-network signature solvers for path-dependent PDEs have been proposed in ~\cite{feng2023,ppde_nn_1,ppde_nn_2,bayraktar2024deep} and extended to kernel frameworks in ~\cite{PPDE_nn_kernel}; both again rely on multiple path realizations. The kernel-collocation methodology underlying our solver traces back to the Kansa method ~\cite{kansa1,kansa2} for RBF-based PDE solving, with subsequent RBF and RKHS methods for differential equations~\cite{RBF_kernel_methods,rkhs_nonlinear_bvps} and Gaussian-process priors for nonlinear PDE solvers~\cite{nonlinearpdes_GP_RKHS}, which inform our regularization and warm-starting choices in the nonlinear case. Finally, the optional rolling prediction protocol developed here is a kernel-method analogue of online updating: we exploit the block structure of the augmented kernel matrix to update only the newest block in closed form, with periodic full-batch retraining to control drift, paralleling recursive least squares and online kernel ridge regression but, to our knowledge, not previously instantiated for signature-kernel ODE solvers.

Based on the above ideas, we develop a branched signature-kernel collocation solver that operates on a \emph{single observed forcing trajectory}, the central novelty of this paper. A \emph{count-sampling} scheme turns one discrete observation into a hierarchical family of
nested training paths $\bff_0\subset\bff_1\subset\cdots\subset\bff_N$, on which the branched signature kernel is natively evaluated and across which Chen's identity yields prefix-sharing speedups; to our knowledge this single-trajectory
training set is new. Building on it, we develop two complementary collocation formulations---a differential ansatz on the highest-order derivative $\bu^{(m)}$ and an integral ansatz on $\bu$ itself, both applicable to multidimensional linear and nonlinear systems. We prove a universal approximation theorem for the branched signature kernel, establishing it as a universal hypothesis class for continuous functionals of the forcing path. Numerically, the branched signature kernel is computed via the Hairer--Kelly morphism \cite{hairer2015geometric}, whose neural network approximation was first proposed in \cite{ali2025branchedsignaturemodel}. The offline solver is extended to a streaming Test/Train/Retrain protocol with either standard Train/Retrain operations or rolling retraining with closed-form $O(nd^{2}+d^{3})$ online updates in the linear case and scalar Newton steps in the nonlinear case.

The paper is organized as follows. Section~\ref{sec: pre} introduces signatures and signature kernels. Section~\ref{sec: brached} presents the branched signature-kernel model and proves a universal approximation theorem. Section~\ref{sec: kernel solvers} develops branched signature-kernel solvers for linear ODE systems, including the \emph{linear integral method} and its $m$-fold-integrated variant, and for nonlinear ODE systems via the \emph{nonlinear integral method}. Section~\ref{sec:test-train} introduces a standard prediction method along with a streaming protocol with closed-form online updates in the linear case and scalar Newton updates in the nonlinear case, and presents a periodic retraining algorithm. Section~\ref{sec:numerical-experiments} presents numerical experiments on an earthquake displacement model, the Solow capital-stock model, an fBM-driven linear ODE, a forced Duffing oscillator, an Arias-intensity-degraded single-degree-of-freedom system with path-dependent coefficients, and a rough Kuramoto phase-oscillator system, with side-by-side comparisons of branched and non-branched constructions.

%%%%%%%%%%%%%%%%%%%%%%%%%%%%%%%%%%
\section{Preliminaries}\label{sec: pre}

In this section, we first provide a self-contained introduction to signatures and signature kernels. We then define the branched signature \cite{ali2025branchedsignaturemodel} and the corresponding branched signature kernels.

\subsection{Paths and Signatures}
To begin with, we will set some terminology that will help us to talk about signatures and signature kernels. 
\begin{defn}
A path \(\bX\) from some time interval \([0,T]\) to \(\cH\) is a continuous mapping written as
\[\bX: [0,T] \to \cH.\]
Also, we define \(\bX_{s,t} = \bX_t - \bX_s\) as an increment of the path.
\end{defn}
Unless otherwise stated, we will be working with the paths that take values in the finite dimensional Hilbert space \(\cH\) with \(\dim(\cH) = d\). Let \(\cA\) be the alphabet set corresponding to \(\cH\) given as \(\cA = \{1,2,\dots,d\}\). We define a word \(\bw\) of length \(|\bw|=k\) to be sequence \(\bw = w_1\cdots w_k\), where \(w_i \in \cA\) for \(i=1,2,\dots,k\). Let \(\bW\) be the set of all words made from the alphabet \(\cA\). Define
\[ 
\bW_n :=
\begin{cases}
\left\{\bw = w_{i_1}\cdots w_{i_k}: |\bw| =  n\right\},  & n \ne 0,\\
\emptyset,   & n=0.
\end{cases}
\]
\begin{defn}
For two words \(\bv = v_1\cdots v_j\) and \(\bw = w_1\cdots w_k\) , we define the concatenation \(\bv \bw\) as 
\[\bv \bw = v_1\cdots v_j w_1\cdots w_k.\]
\end{defn}
The signature of a path defined on \(\cH\) lives in the space called \textbf{extended tensor algebra} over \(\cH\) defined as
\[T((\cH)) =  \prod_{n=0}^{\infty} (\cH)^{\otimes n} = \{\bv = (v_0, v_1, \dots , v_n, \dots ) \; | \; v_n \in (\cH)^{\otimes n}, n = 0, 1,\dots \},\]
where \((\cH)^{\otimes 0} := \RR\). We also define the \textbf{tensor algebra} over \(\cH\) to be the space \(T(\cH)\) defined by 
\[T(\cH) =  \bigoplus_{n=0}^{\infty} (\cH)^{\otimes n} = \{ \bv \in T((\cH)) \; | \; \forall \;  \bv \; \exists \;  K \in \NN \text{ such that } v_n = 0 \; \forall \; n \ge K\},\]
where most of the components of the infinite dimensional object \(\bv \in T((\cH))\) are zero. If we want to restrict the non-zero components then we have the \textbf{truncated tensor algebra} over \(\cH\) which is defined as
\[T^N(\cH) := \left\{ \mathbf{v} \in T(\cH) \; | \; v_n = 0 \; \forall \;  n > N\right\}.\]
Next, define paths of bounded \(p\)-variation for \(p=1\) also called paths of bounded variation.
\begin{defn}
Let $\mathcal{H}$ be a Hilbert space. We denote the set of $\mathcal{H}$-valued paths of bounded variation on $[0,1]$ starting at the origin by
$$
C^1([0,1], \mathcal{H}):=\left\{\bX \in C([0,1], \mathcal{H}): \bX(0)=0, \quad\|\bX\|_1<\infty\right\}
$$
where $\|\bX\|_1=\sup _\pi \sum_{i=1}^{l-1}\left\|\bX\left(t_{i+1}\right)-\bX\left(t_i\right)\right\|_{\mathcal{H}}$ and the supremum is taken over all finite partitions $\pi=\left\{\left(t_i\right)_{i=1, \ldots, l}: 0=t_1<\cdots<t_l=1\right\}$ of $[0,1]$.
\end{defn}
Though we will be working with paths of bounded \(p\)-variation for any \(p \ge 1\), the signature of bounded \(p\)-variation for \(p = 1\) is canonically defined using Riemann--Stieltjes integration. 
\begin{defn}
The signature of an \(\cH\)-valued path of bounded variation \(\left(\bX_t\right)_{t \in [0,T]}\) is an infinite dimensional object \(\sig(\bX)_{st}\) that belongs to the space \(T((\cH))\) and is defined as
\[\sig(\bX)_{st} = \left(1, \int_s^t d\bX^{w_i}_r, \int_s^t \int_s^{r_2} d\bX^{w_i}_{r_1} d\bX^{w_j}_{r_2}, \int_s^t \int_s^{r_3}\int_s^{r_2} d\bX^{w_i}_{r_1} d\bX^{w_j}_{r_2} d\bX^{w_k}_{r_3}, \cdots \right)\]
Where ${w_i,w_j,w_k,  \dots \in \cA}$. Equivalently, the signature of a continuous \(\cH\)-valued path of bounded variation \((\bX_t)_{t \in [0,T]}\) is a \(T((\cH))\)-valued process \( (s,t) \in \Delta_T^2 \mapsto \sig(\bX)_{st} \in T((\cH))\) also defined recursively as
\[\innerprod{\sig(\bX)_{st}, \emptyset} := 1, \qquad \innerprod{\sig(\bX)_{st}, \bw} := \int_{s}^{t} \innerprod{\sig(\bX)_{sr}, w_1\cdots w_{n-1}} d\bX^{w_n}_{r},\]
for each word \(\bw = w_1 \dots w_n \in \bW\) and \(\Delta_T^2 := \{ (s,t) \in [0,T]^2 : 0 \leq s \leq t \leq T \}\).
\end{defn}

The signature of a path of bounded variation enjoys several structural properties, one being the shuffle product. To describe the shuffle property for signature components, we first define the shuffle product on words \(\bw\) in \(\bW\).

\begin{defn}
Let \(\emptyset\) denote the empty word. For words \(\bu,\bv\) and letters
\(a,b\), the shuffle product is defined recursively by
\[
\bu \shuffle \emptyset = \emptyset \shuffle \bu = \bu,
\qquad
(\bu a) \shuffle (\bv b)
=
\bigl(\bu \shuffle \bv b\bigr)a
+
\bigl(\bu a \shuffle \bv\bigr)b .
\]
\end{defn}
Fix a basis \(\{e_1,\dots,e_d\}\) of \(\cH\), and for a word \(\bw=w_1\cdots w_n\) define
\(e_{\bw}:=e_{w_1}\otimes\cdots\otimes e_{w_n}\) and \(e_{\emptyset}:=1\). The shuffle product then extends from words to elements of \(T(\cH)\) by
\[
\bx
\shuffle
\by
=
\sum_{\bu,\bv}x_{\bu}y_{\bv}
\left(e_{\bu}\shuffle e_{\bv}\right),
\]
where \(\bx=\sum_{\bu}x_{\bu}e_{\bu}, \by=\sum_{\bu}y_{\bu}e_{\bu} \in T(\cH)\). Following this definition, if we endow the space \(T(\cH)\) with the shuffle product \(\shuffle\) then the quadruple \(\left(T(\cH), +, \cdot, \shuffle \right)\) is a commutative associative algebra.

\begin{defn}
Let \((\bX_t)_{t \in [0,T]}\) be a continuous \(\cH\)-valued path of bounded variation and let \(\bu = u_1 \dots u_n\) and \(\bv = v_1 \dots v_m\) be two words. The shuffle property states that 
\[
\innerprod{\sig(\bX)_{st}, \bu}\,
\innerprod{\sig(\bX)_{st}, \bv}
=
\innerprod{\sig(\bX)_{st}, \bu \shuffle \bv}.
\]
\end{defn}

\begin{rem}
When \(|\bu| = |\bv| = 1\), the shuffle property reduces to the classical integration-by-parts identity. See for example, \cite{lyons2007differential}.
\end{rem}

Another nice property is Chen's Identity (See \cite{Friz_Victoir_2010}). This allows for the concatenation of paths, which we will rely on later in our count method speedup. 
\begin{prop}
Let \((\bX_t)_{t \in [0,T]}\) be a continuous \(\cH\)-valued path of bounded variation. Then the concatenated signature over intervals \([s,u]\) and \([u,t]\) satisfies
\[
\sig(\bX)_{st} = \sig(\bX)_{su} \otimes \sig(\bX)_{ut},
\qquad 0 \leq s \leq u \leq t \leq T. 
\]
Equivalently, for every word \(\bw \in \bW\), denote \(\bu \bv\) as concatenation. We have
\[
\innerprod{\sig(\bX)_{st}, \bw}
=
\sum_{\bu \bv = \bw}
\innerprod{\sig(\bX)_{su}, \bu}\,
\innerprod{\sig(\bX)_{ut}, \bv}.
\]
\end{prop}

In general, if the signatures of two paths of bounded variation coincide, then the paths agree up to tree-like equivalence \cite{hambly2010uniqueness}. Uniqueness of signatures for time-extended semimartingales was proved in \cite{cuchiero2023signature}. We will use the following uniqueness theorem for geometric rough paths of time-extended paths (proved in \cite{ali2025branchedsignaturemodel}), which is key to using the signature as a feature map that can universally approximate functionals of the underlying path under suitable assumptions.

\begin{lem} \label{sig_uiqueness}
Let $ \bX,\bY:[0,T]\to\mathbb R^d$ be continuous $\alpha$-H\"older paths with $\bX_0=\bY_0=0$ for some $\alpha >  \frac{1}{4}$.
Form the time-augmented paths $\widehat{\bX}(t):=(t,\bX_t)$ and $\widehat{\bY}(t):=(t,\bY_t)$ in $\mathbb R^{1+d}$.
Assume their (geometric) terminal signatures coincide at all levels i.e.,
\[
\sig(\widehat{\bX})_{0,T}=\sig(\widehat{\bY})_{0,T}.
\]
Then $\bX_t=\bY_t$ for all $t\in[0,T]$.
\end{lem}

\subsection{Kernels and Signature Kernels}

Before defining a signature kernel, we recall standard static kernels and their key properties, since the signature kernel will simply be a kernel on path space.

\begin{defn}
Let $\cX$ be a nonempty set. A function $\bK:\cX\times\cX\to\mathbb{R}$ is a
\emph{kernel} if there exists a Hilbert space $\cH$ and a map $\phi:\cX\to\cH$
such that
\[
\bK(x,x') := \bigl\langle \phi(x),\phi(x') \bigr\rangle_{\cH},
\qquad \forall\, x,x'\in\cX.
\]
A kernel $\bK$ is \emph{positive definite} if it is symmetric and, for every
$n\in\mathbb{N}$, every $x_1,\ldots,x_n\in\cX$, and every
$c_1,\ldots,c_n\in\mathbb{R}$,
\[
\sum_{i=1}^n\sum_{j=1}^n c_i c_j\,\bK(x_i,x_j) > 0.
\]
\end{defn}

This definition ensures that \(\bK\) can be realized as an inner product in some feature space \(\cH\), which is exactly the structure we will exploit on path space. Two common example kernels are as follows:

\begin{example}
The linear kernel between a set of data points $x,y \in \mathbb{R}^d$ is simply the inner product of the two inputs
\[
\bK(x,y) = \bigl <x,y\bigr>_{\RR^d}.
\]
\end{example}

\begin{example}
For \(x,y \in \mathbb{R}^d\), we define the RBF kernel using the squared
Euclidean distance \(\|x-y\|_2^2\) and a positive parameter \(\sigma>0\) as follows:
\[
\bK(x,y) = \exp\left(-\frac{\|x-y\|_2^2}{2\sigma^2}\right).
\]
\end{example}

The Moore–Aronszajn theorem explains how any positive definite kernel induces a canon\-ical function space, the RKHS, in which kernel methods act as linear models in feature space.

\begin{thm}
For every symmetric positive definite kernel $\bK:\cX\times\cX\to\RR$,
there exists a unique Hilbert space $\cH_{\bK}$ of functions $f:\cX\to\RR$,
the \emph{reproducing kernel Hilbert space (RKHS)} of $\bK$,
satisfying the reproducing property
\[
    \langle f,\,\bK(x,\cdot)\rangle_{\cH_{\bK}} = f(x),
    \qquad \forall\,f\in\cH_{\bK},\; x\in\cX.
\]
\end{thm}

Within an RKHS, the minimum‑norm interpolant formalizes the basic kernel‑collocation idea we will later use with signature kernels.

\begin{defn}
Given distinct points $x_1,\ldots,x_N\in\Omega$ with values $y_1,\ldots,y_N\in\RR$
and an RKHS $\cH$ over $\Omega$, the \emph{minimum-norm interpolant} is
\[
f^*_{X,Y} = \argmin_{f\in\cH}\|f\|_{\cH}
\quad\text{subject to}\quad f(x_i)=y_i,\quad i=1,\ldots,N.
\]
\end{defn}

In our setting, we will construct such interpolants on a space of paths, using a kernel built from path signatures.

\begin{defn}\label{def:linsigkernel}
Let \((\bX_t)_{t \in [0,T]}\) and \((\bY_t)_{t \in [0,T]}\) take values in \(\cX\), and let
\(\bK:\cX\times\cX\to\RR\) be a positive definite kernel with associated RKHS \(\cH\).
We lift each path \(\bX,\bY:[0,T]\to\cX\) to the \(\cH\)-valued paths
\[
    \bK_X(t):=\bK(X_t,\cdot)\in\cH,
    \qquad
    \bK_Y(t):=\bK(Y_t,\cdot)\in\cH.
\]
Writing \(\sig(\bK_X)\) and \(\sig(\bK_Y)\) for the signatures of the lifted paths,
the \emph{signature kernel} is
\[
\bK_{\sig}(\bX,\bY)
:=
\bigl\langle \sig(\bK_X),\sig(\bK_Y)\bigr\rangle_{T((\cH))},
\]
where the inner product is taken over the extended tensor algebra \(T((\cH))\) \cite{kiraly2019kernels}.
\end{defn}

\begin{rem}
If the paths \(\bX_t\) and \(\bY_t\) take values in \(\RR^d\), we may take
\(\cX=\RR^d\) and \(\bK(x,x')=\langle x,x'\rangle_{\RR^d}\), so that
\[
    \bK_{\sig}(\bX,\bY)
    =
    \bigl\langle \sig(\bX),\sig(\bY)\bigr\rangle_{T((\RR^d))}.
\]
Unless otherwise stated, all paths in our consideration will take values in \(\RR^d\) for some \(d \ge 1\). For our numerical experiments we use the truncated signature kernel, with the scalar product of the signatures truncated at level \(N\):
\[
    \bK^N_{\sig}(\bX,\bY)
    :=
    \bigl\langle \sig^N(\bX),\sig^N(\bY)\bigr\rangle_{T^N(\RR^d)},
\]

\end{rem}

Besides the linear signature kernel, we will also use a Gaussian (RBF) kernel on signature features. 

\begin{defn}[Truncated RBF Signature Kernel]\label{def:rbfsigkernel}
Let \(\bX\) and \(\bY\) be paths in \(\RR^d\), and let \(\sig^m(\bX)\) and \(\sig^m(\bY)\) denote their signatures truncated up to level \(m\). After flattening the truncated signatures into vectors, we define the truncated RBF signature kernel by
\[
\bK_{\mathrm{RBF\text{-}sig}}^{m}(\bX,\bY) = \exp\left( -\frac{ \left\| \sig^m(\bX) - \sig^m(\bY) \right\|_2^2 }{2\sigma^2} \right), \qquad \sigma>0.
\]
Equivalently, using the Euclidean inner product, we can write
\[
\bK_{\mathrm{RBF\text{-}sig}}^{m}(\bX,\bY) = \exp\left( -\frac{ \left\| \sig^m(\bX) \right\|_2^2 + \left\| \sig^m(\bY) \right\|_2^2  -2 \innerprod{\sig^m(\bX), \sig^m(\bY)}}{2\sigma^2} \right).
\]
\end{defn}

\subsection{Kernel Matrices and Signature Normalizations}

We next describe simple normalizations of signature features and the associated kernel Gram matrices used in our experiments. Throughout this subsection, let \(\bX_1,\ldots,\bX_n\) be paths in \(\RR^d\), let
\(S_i := \sig^m(\bX_i) \in \RR^M\) denote their flattened truncated signatures, where \(M = \sum_{k=0}^{m} d^k\), and let \(S \in \RR^{n\times M}\)
be the matrix with \(S_i\) as its \(i\)-th row.

\begin{defn}[Robust signature matrix]
The \emph{robust signature matrix} $S_{\mathrm{robust}} \in \RR^{n\times M}$
is defined coordinate-wise by
\[
(S_{\mathrm{robust}})_{ij}
= \frac{S_{ij} - \mu_j}{\operatorname{IQR}_j},
\qquad i=1,\ldots,n,\quad j=1,\ldots,M,
\]
where $\mu_j$ and $\operatorname{IQR}_j = Q_{0.75}-Q_{0.25}$ are the column-wise median and inter-quartile range of $S$.
\end{defn}

This transformation rescales each signature coordinate using robust statistics, improving conditioning when constructing the following gram matrices.

\begin{defn}
The \emph{linear signature kernel Gram matrix} $\KK = SS^\top \in \RR^{n\times n}$
has entries
\[
\KK_{ij} = \innerprod{\sig^m(\bX_i), \sig^m(\bX_j)},
\qquad i,j=1,\ldots,n.
\]
\end{defn}

Thus \(\KK\) is simply the Gram matrix of the truncated signature feature vectors under the Euclidean inner product. We may instead apply the Gaussian kernel, providing a nonlinear similarity measure between paths.

\begin{defn}
The \emph{RBF signature kernel Gram matrix} $\KK \in \RR^{n\times n}$
has entries
\[
\KK_{ij} = \exp\!\left(
-\frac{\|S_i - S_j\|_2^2}{2\sigma^2}
\right),
\qquad i,j=1,\ldots,n,\quad \sigma>0.
\]
\end{defn}

\subsection{Sampling methods and Count Method Speedup}
We now describe a novel way to turn a single discretized path into a nested family of training paths and explain how this structure accelerates signature and signature-kernel computations.

\begin{defn}[Count-sampling method]\label{defn:count-sampling}
Let \(\bX = (X^1, \ldots, X^d) : [0, T] \to \R^d\) be a path discretized at \(N\) points \(0 = t_0 < t_1 < \cdots < t_{N-1} = T\), and write \(\bX_{0,\, t_j}\) for the restriction of \(X\) to the prefix interval \([0, t_j]\). The count-sampling construction produces the family of \(N\) nested training paths
\[
\bX_0 := \bX_{t_0,\, t_0},\qquad \bX_1 := \bX_{0,\, t_1},\qquad \ldots,\qquad \bX_{N} := \bX_{0,\, T}.
\]
The first sample \(\bX_0\) is the degenerate two-point path \((\bX(0), \bX(0))\), so that \(\bX_0\) has at least two values and the signature kernel is well-defined on it. This convention lets the initial point be treated uniformly with the longer prefixes and is used by the algorithms below.
\end{defn}

\begin{rem}\label{rmk:count-sampling}
All count-sampled paths \(\bX_0, \ldots, \bX_{N-1}\) share the starting point \(\bX(0)\) but differ in length: \(\bX_j\) is supported on \([0, t_j]\) and contains strictly more path information than its predecessors. From a single observed time series we therefore obtain a hierarchical training set of \(N\) nested paths, ideally suited as inputs to the signature kernel \(\bK_{\sig}\). Throughout the rest of the section, “training paths \(\bX_0, \ldots, \bX_{N-1}\)” refers to this construction.
\end{rem}

When using the count-sampling method, one can exploit Chen's identity to greatly reduce computational cost. Two popular signature libraries, \texttt{keras-sig} and \texttt{signatory}, expose this speedup through a \texttt{stream} argument and support GPU acceleration; we rely on these implementations in our numerical experiments.

Suppose we use the count-sampling method to generate \(N\) nested training paths \(\bX_0, \bX_1, \dots, \bX_{N-1}\), each of dimension \(d\). The \(j\)-th path has length
\[
L_j =
\begin{cases}
2, & j=0,\\
j+1, & j \ge 1,
\end{cases}
\]
with \(L_0 = 2\) ensuring that even the shortest path admits a well-defined truncated signature; this convention is immaterial for the asymptotics. (Here \(N\) is the number of count-sampled paths in this section; this corresponds to \(N{+}1\) in the collocation setup, where paths were indexed \(\bX_0,\ldots,\bX_N\).)

In most numerical libraries, path tensors have shape \((B,L,D)\), where \(B\) is the batch size, \(L\) the path length, and \(D\) the path dimension. The naive approach therefore produces \(N\) separate tensors of size \((1, L_j, d)\) and calls the signature routine on each independently. By the analysis of \cite{kidger2021signatory}, computing the truncated signature of a single length-\(L\) path to depth \(M\) costs \((L-1)\cdot O(d^M)\), so summing over all prefixes yields a total cost of order
\[
O\!\left(\sum_{j=0}^{N-1} (L_j - 1)\,d^M\right)
=
O\!\left(\left(1 + \tfrac{N(N-1)}{2}\right)d^M\right)
=
O(N^2 d^M).
\]

With the count-method speedup, we instead pass only the final prefix tensor of size \((1,N,d)\) to the signature function. It processes the entire stream from left to right, and Chen's identity
\[
\sig^M(\bX_{[t_0,t_{j+1}]}) = \sig^M(\bX_{[t_0,t_j]}) \otimes \sig^M(\bX_{[t_j,t_{j+1}]})
\]
allows us to reuse all previously computed leading increments. The total time complexity becomes
\[
(N-1)\cdot O(d^M) = O(N d^M),
\]
an order-\(N\) improvement over the naive \(O(N^2 d^M)\) scaling. These methods are displayed in Table~\ref{tab:sig_runtime}.

\begin{table}[h]
\centering
\begin{tabular}{lc}
\hline
\textbf{Method} & \textbf{Time} \\
\hline
Naive (independent calls) & $O(N^2 d^M)$  \\
Count method       & $O(N d^M)$  \\
\hline
\end{tabular}
\caption{Runtime comparison for computing all $N$ truncated signature vectors 
to depth $M$ for the count-sampled prefix paths $\bX_0, \ldots, \bX_{N-1}$ 
of dimension $d$. The naive case calls \texttt{signature()} independently 
on each path tensor of size $(1, L_j, d)$. The stream case passes a single 
tensor of size $(1, N, d)$ with \texttt{stream=True}.}
\label{tab:sig_runtime}
\end{table}

Now consider computing the truncated signature-kernel Gram matrix of the prefix paths up to depth \(M\). Previous methods either rely on dynamic-programming techniques \cite{kiraly2019kernels} or on PDE approximations exploiting the fact that the signature kernel satisfies a Goursat PDE \cite{Salvi_2021_goursat}; both are implemented efficiently in the Python package \texttt{PySigLib}. The Goursat approach has cost \(O(N^2 L^2 d)\), and since the maximum path length satisfies \(L = N\) in our count-sampled setting, this becomes \(O(N^4 d)\). A dynamic-programming approach yields \(O(N^2 L^2 M d) = O(N^4 M d)\) for the same \(N \times N\) kernel matrix. 

With the count speedup, we first compute the truncated signatures in \(O(N d^M)\) time, producing an \(N \times C\) matrix of signature vectors \(S\), where row \(j\) holds \(\sig^M(\bX_j)\) flattened into a vector of length \(C = \sum_{k=0}^{M} d^k = O(d^M)\). To form the kernel Gram matrix we then compute \(K = S S^{\top}\), a matrix multiplication costing \(O(N^2 d^M)\). The overall runtime is therefore \(O(N^2 d^M)\), in which the matrix multiplication dominates the signature-computation step. Runtimes for all methods are summarized in Table~\ref{tab:gram_runtime}.

\begin{table}[h]
\centering
\begin{tabular}{lc}
\hline
\textbf{Method} & \textbf{Time} \\
\hline
Goursat PDE  &  $O(N^4 d)$ \\
Dynamic programming  & $O(N^4 M d)$ \\
Count method & $O(N^2 d^M)$ \\
\hline
\end{tabular}
\caption{Runtime comparison for computing the $N \times N$ truncated signature 
kernel Gram matrix $\mathbb{K}_{ij} = \innerprod{ \mathrm{\sig}^M(\bX_i), \mathrm{\sig}^M(\bX_j)}$ 
for the count-sampled prefix paths. $N$ denotes the number of paths, $L \leq N$ the 
maximum path length, $d$ the path dimension, and $M$ the truncation depth. (Note L=N)}
\label{tab:gram_runtime}
\end{table}

The Goursat and dynamic-programming methods inherit only a linear $d$-dependence from the static-kernel evaluations they invoke.  The count method, by contrast, materializes the full truncated signature vector of length $O(d^M)$ and pays for it explicitly in the Gram matrix multiplication.  The two regimes therefore exchange a quartic $N^4$ scaling for an exponential $d^M$ scaling --- a favorable trade-off whenever
\begin{equation*}
N \;\gtrsim\; d^{(M-1)/2},
\end{equation*}
obtained by equating $N^4 d$ with $N^2 d^M$.  Our experiments operate in the regime where $N$ is large and $M$ is small --- the branched signature framework lets us extract sufficient path-wise information at low truncation depth --- so the $N^4 \to N^2$ improvement dominates the unfavorable $d \to d^M$ trade-off.  This regime is exactly the one encountered in time-series data-driven modeling, where long observation records produce many nested prefix paths but only modest signature depth is needed, and the count-method speedup is therefore decisive.

\subsection{Signature kernel for geometric rough path}
In this subsection, we extend the signature kernel from paths of bounded variation to geometric rough paths. In general, \(\alpha\)-H\"older paths with \(\alpha > \frac{1}{4}\) admit geometric rough path lifts, and by Lyons' extension theorem \cite{Friz_Victoir_2010}, the signatures of these lifts are well-defined at all higher levels. This allows us to define the signature kernel for such geometric rough paths as follows, see \cite{kiraly2019kernels}.

\begin{defn}
Let \((\bX_t)_{t \in [0,T]}\) and \((\bY_t)_{t \in [0,T]}\) be \(\alpha\)-H\"older paths with \(\alpha > \tfrac{1}{4}\) taking values in \(\RR^d\), and let \(\sig(\bX)\) and \(\sig(\bY)\) be their signatures defined by Lyons' extension theorem. The signature kernel \(\bK_{\sig}(\bX,\bY)\) is defined for these paths as
\[
    \bK_{\sig}(\bX,\bY)
    :=
    \bigl\langle
    \sig(\bX),\sig(\bY)
    \bigr\rangle_{T((\RR^d))},
\]
where
\[
    \bigl\langle
    \sig(\bX),\sig(\bY)
    \bigr\rangle_{T((\RR^d))}
    :=
    \sum_{n=0}^{\infty}
    \bigl\langle
    \pi_n\bigl(\sig(\bX)\bigr),\pi_n\bigl(\sig(\bY)\bigr)
    \bigr\rangle_{(\RR^d)^{\otimes n}},
\]
whenever the series is well-defined. Here \(\pi_n(\sig(\bX))\in(\RR^d)^{\otimes n}\) denotes the level-\(n\) signature component of \(\bX\).
Equivalently, if \(\bw = w_1\dots w_n\) is a word of length \(n\) from the set of words \(\bW\), then
\[
\bK_{\sig}(\bX,\bY)
=
\sum_{n=0}^{\infty}
\sum_{|\bw|=n}
\langle \sig(\bX),\bw\rangle\,
\langle \sig(\bY),\bw\rangle.
\]
\end{defn}

We now state and prove a universal approximation theorem for the signature kernel on geometric rough paths by combining this definition with the universal approximation theorem for signatures of geometric rough paths.

\begin{thm}[Universal approximation theorem for signature kernel]
Let \(\cK\) be the compact set of paths that admit a geometric rough path lift i.e.,
\[
\cK \subset \left\{\bX: [0,T] \to \RR^d \text{ such that } \bX \text{ admits a geometric rough path lift}\right\}.
\]
Define the compact set of time extended paths \(\widehat{\cK}\) as 
\[
\widehat{\cK} \subset \left\{\widehat{\bX}_t:= (t, \bX_t) \text{ such that } \bX_t \in \cK \right\}.
\]
Then for every continuous function \(f(\widehat{\cK}; \RR)\) and every \(\epsilon > 0\), there exist an integer \(N \ge 0\), paths \(\widehat{\bY}^1, \dots, \widehat{\bY}^M \in \widehat{\cK}\), and scalars \(\alpha_1, \dots \alpha_M\) such that
\[
\sup_{\widehat{\bX} \in \widehat{\cK}}
\left|
f\left(\sig^p(\widehat{\bX})\right)
-
\sum_{i=1}^{M}
\alpha_i
\bK^N_{\sig}(\widehat{\bX},\widehat{\bY}^i)
\right|
<\varepsilon,
\]
where \(\sig^p(\widehat{\bX})\) is the geometric rough path lift of \(\widehat{\bX}\) for some \(p \in \NN\) depending on the H\"older regularity of \(\bX\).
\end{thm}

\begin{proof}
Leveraging the universal approximation theorem for signatures of geometric rough paths \cite{ali2025branchedsignaturemodel}, we have that for every \(f\in C(\widehat{\cK};\RR)\) and every \(\varepsilon>0\), there exist
\(N\geq 0\) and \(\ell\in T^{(N)}(\RR^{d+1})\) such that
\begin{equation} \label{sig UAT}
\sup_{\widehat{\mathbf X}\in\widehat{\cK}}
\left|
f\left(\sig^p(\widehat{\bX})\right)
-
\left\langle
\sig^N(\widehat{\mathbf X})
, \ell \right\rangle
\right|
<\varepsilon.
\end{equation}
Now, let
\[
M_N
:=
\operatorname{span}
\left\{
\sig^N(\widehat{\mathbf Y})
:
\widehat{\mathbf Y}\in\widehat{\cK}
\right\}
\subset T^{(N)}(\RR^{d+1}).
\]
Let \(\cP_{M_N}\) be the orthogonal projection onto \(M_N\), and define
\[
\ell_N:=\cP_{M_N}\ell.
\]
For every \(\widehat{\mathbf X}\in\widehat{\cK}\), we have
\[
\sig^N(\widehat{\mathbf X})\in M_N.
\]
Therefore,
\[
\left\langle \sig^N(\widehat{\mathbf X}),\ell \right\rangle
=
\left\langle \sig^N(\widehat{\mathbf X}), \cP_{M_N}\ell \right\rangle
=
\left\langle \sig^N(\widehat{\mathbf X}), \ell_N \right\rangle.
\]
Since \(\ell_N\in M_N\), by the definition of \(M_N\), there exist \(\widehat{\mathbf Y}^{1},\ldots,\widehat{\mathbf Y}^{M}\in\widehat{\cK}\) and scalars \\ \(\alpha_1,\ldots,\alpha_M\in\RR\) such that
\[
\ell_N
=
\sum_{i=1}^{M}
\alpha_i
\sig^N(\widehat{\mathbf Y}^{i}).
\]
Hence, for every \(\widehat{\mathbf X}\in\widehat{\cK}\),
\[
\begin{aligned}
\left\langle
\sig^N(\widehat{\mathbf X}), \ell
\right\rangle
&=
\left\langle
\sig^N(\widehat{\mathbf X}), \ell_N
\right\rangle 
=
\left\langle
\sig^N(\widehat{\mathbf X}), 
\sum_{i=1}^{M}
\alpha_i
\sig^N(\widehat{\mathbf Y}^{i})
\right\rangle \\
&=
\sum_{i=1}^{M}
\alpha_i
\left\langle
\sig^N(\widehat{\mathbf X}),
\sig^N(\widehat{\mathbf Y}^{i})
\right\rangle 
=
\sum_{i=1}^{M}
\alpha_i
\bK_{\sig}^{N}(\widehat{\mathbf X},\widehat{\mathbf Y}^{i}).
\end{aligned}
\]
Combining this identity with ~\eqref{sig UAT} gives
\[
\sup_{\widehat{\bX} \in \widehat{\cK}}
\left|
f\left(\sig^p(\widehat{\bX})\right)
-
\sum_{i=1}^{M}
\alpha_i
\bK^N_{\sig}(\widehat{\bX},\widehat{\bY}^i)
\right|
<\varepsilon.
\]
This proves the theorem.
\end{proof}

%%%%%%%%%%%%%%%%%%%%%%%%%%%%%%%%%%
\section{Branched Signature Kernel}\label{sec: brached}
In this section we first define the branched signature and the branched signature kernel. We then describe some avaliable ways to construct the branches signature components though lifting paths.

\subsection{Branched Signature Kernel}\label{subsec: brached}
To begin with, let us define a Hopf algebra and Connes–Kreimer Hopf algebra of rooted trees, a special example of Hopf algebra where branched signature lives. 

A Hopf algebra \(\hopf\) is an algebraic structure that contains a product \(\cdot : \hopf \hat{\otimes} \hopf \to \hopf \), which tells how to combine two objects; a coproduct \(\Delta: \hopf \to \hopf \hat{\otimes} \hopf\), which tells how to split an object; and an antipode \(\cS :  \hopf \to \hopf\), which is the way to encode inverses. Let \(\hopf^*\) be the dual of \(\hopf\). The product \(\star\) on the dual space \(\hopf^*\) is defined via coproduct on \(\hopf\) as
\[ \innerprod{f \star g, h} = \innerprod{f \; \hat{\otimes}\; g, \Delta h},\]
for any \(f,g \in \hopf^*\) and \(h \in \hopf\). 

The Connes--Kreimer Hopf algebra of rooted trees is a special type of Hopf algebra which is used in the setting of branched rough paths and branched signatures. The iterated integrals for geometric rough paths are encoded by words \(\bw\) made from the alphabet \(\cA = \{1,2,\dots,d\}\). However, in the case of branched rough paths, the iterated integrals are encoded by decorated rooted trees, with decorations coming from the set \(\cA\). For instance, the iterated integral
\[
\int_s^t \left(\int_s^u d\bX^{i}_r\right) \left(\int_s^u d\bX^{j}_r\right) d\bX^{k}_u
\]
is encoded by the tree
\[
\Forest{decorated[k[i][j]]}.
\]
The Connes--Kreimer Hopf algebra of rooted trees \(\hopf\) \cite{connes1999hopf} is the polynomial algebra in which rooted trees play the role of variables. The product of trees is given by placing trees side by side to form a forest, while the coproduct is given by all possible ways of cutting a tree into smaller trees. The Connes--Kreimer Hopf algebra of rooted trees \(\hopf\) is the appropriate algebraic structure in which the branched signature lives. The definition of the branched signature follows from \cite{gubinelli2010ramification}.

\begin{defn}
 Let $\cA=\{1,\dots,d\}$ be an alphabet set for a given \(d\)-dimensional path \(\bX: [0,T] \to \RR^d\). Let $\trees$ be the set of rooted trees with vertices decorated by letters in $\cA$. Let $\hopf$ be the Connes–Kreimer Hopf algebra of rooted trees generated by trees is $\trees$, with product given by disjoint union of forests and unit $\mathbf 1$. We define \emph{branched signature} of \(\bX\) as a functional on \(\hopf\) given by
\begin{equation}
\bsig(\bX)_{st} =  \sum\limits_{\tau \; \in \; \mathcal T, \; |\tau| \; \le \; N} \innerprod{\bsig(\bX)_{st},\tau} \se_{\tau},
\end{equation}
where for each \(\tau \in \hopf\) the component \(\innerprod{\bsig(\bX)_{st},\tau}\) of the branched signature is recursively defined as
\[\innerprod{\bsig(\bX)_{st},\mathbf{1}} = 1, \quad \text{ and } \quad \innerprod{\bsig(\bX)_{st},\tau} = \int_s^t \innerprod{\bsig(\bX)_{su},\tau'} d\bx^{\mathbf{r}}_u,\]
where \(\tau\) is the tree that we get by grafting the root vertex \(\mathbf{r}\) to  \(\tau'\) i.e.,  \(\tau = [\tau']_{\mathbf{r}}\).
\end{defn}
\begin{defn}[Branched signature kernel]
Let \(\left(\bX_t\right)_{t \in [0,T]}\) and \(\left(\bY_t\right)_{t \in [0,T]}\) be \(\alpha\)-H\"older paths for
\(\alpha > \frac{1}{4}\) taking values in \(\RR^d\). Let \(\bsig(\bX)\) and \(\bsig(\bY)\) denote their branched signatures, whose components are indexed by decorated rooted forests in the Connes--Kreimer Hopf algebra \(\hopf\). The branched signature kernel \(\bK_{\bsig}(\bX,\bY)\) is defined as
\[
    \bK_{\bsig}(\bX,\bY) := \left\langle \bsig(\bX),\bsig(\bY) \right\rangle_{\hopf^*},
\]
where
\[
    \left\langle
    \bsig(\bX),\bsig(\bY)
    \right\rangle_{\hopf^*}
    :=
    \sum_{n=0}^{\infty}
    \left\langle
    \pi_n\left(\bsig(\bX)\right),
    \pi_n\left(\bsig(\bY)\right)
    \right\rangle_{\hopf_{n}^*},
\]
whenever the series makes sense. Also,
\[
    \pi_n\left(\bsig(\bX)\right)\in \hopf_{n}^*
\]
denotes the degree-\(n\) component of the branched signature and \(\hopf_{n}^*\) denotes the dual of the space spanned by decorated rooted forests with exactly \(n\) vertices.

Equivalently, if \(\cF_n\) denotes the set of decorated rooted forests of degree \(n\), then
\[
    \bK_{\bsig}(\bX,\bY)
    =
    \sum_{n=0}^{\infty}
    \sum_{\tau\in \cF_n}
    \left\langle \bsig(\bX),\tau\right\rangle
    \left\langle \bsig(\bY),\tau\right\rangle.
\]
\end{defn}
We prove the universal approximation theorem for the branched signature kernel below.
\begin{thm} \label{UAT_bsig_kernel}
Let \(\cK\) be the compact subset of paths that admit a branched rough path lift,
\[
\cK \subset \left\{\bX: [0,T] \to \RR^d \text{ such that } \bX \text{ admits a branched rough path lift}\right\}.
\]
Define the compact set of time extended paths \(\widehat{\cK}\) as 
\[
\widehat{\cK} \subset \left\{\widehat{\bX}_t:= (t, \bX_t) \text{ such that } \bX_t \in \cK \right\}.
\]
Then for every continuous function \(f(\widehat{\cK}; \RR)\) and every \(\epsilon > 0\), there exist an integer \(N \ge 0\), paths \(\widehat{\bY}^1, \dots, \widehat{\bY}^M \in \widehat{\cK}\), and scalars \(\alpha_1, \dots \alpha_M\) such that
\[
\sup_{\widehat{\bX} \in \widehat{\cK}}
\left|
f\left(\bsig^p(\widehat{\bX})\right)
-
\sum_{i=1}^{M}
\alpha_i
\bK^N_{\bsig}(\widehat{\bX},\widehat{\bY}^i)
\right|
<\varepsilon,
\]
where \(\bsig^p(\widehat{\bX})\) is the branched rough path lift of \(\widehat{\bX}\) for some \(p \in \NN\) depending on the H\"older regularity of \(\bX\).
\end{thm}

\begin{proof}
Using the universal approximation theorem for signatures of branched rough paths \cite{ali2025branchedsignaturemodel}, we have that for every \(f\in C(\widehat{\cK};\RR)\) and every \(\varepsilon>0\), there exist
\(N\geq 0\) and \(\tau\in \hopf\) such that
\begin{equation} \label{bsig UAT}
\sup_{\widehat{\mathbf X}\in\widehat{\cK}}
\left|
f\left(\bsig^p(\widehat{\bX})\right)
-
\left\langle
\bsig^N(\widehat{\mathbf X})
, \tau \right\rangle
\right|
<\varepsilon,
\end{equation}
where \(\hopf\) is the Hopf algebra of decorated rooted trees with the decoration coming from \(\cA \cup \{0\}\) with \(\cA := \{ 1,2,\dots,d\}\) and \(0\) corresponding to the time component.
Now, let
\[
M_N
:=
\operatorname{span}
\left\{
\bsig^N(\widehat{\mathbf Y})
:
\widehat{\mathbf Y}\in\widehat{\cK}
\right\}
\subset \hopf^*.
\]
Let \(\cP_{M_N}\) be the orthogonal projection onto \(M_N\), and define
\[
\tau_N:=\cP_{M_N}\tau.
\]
For every \(\widehat{\mathbf X}\in\widehat{\cK}\), we have
\[
\bsig^N(\widehat{\mathbf X})\in M_N.
\]
Therefore,
\[
\left\langle \bsig^N(\widehat{\mathbf X}),\tau \right\rangle
=
\left\langle \bsig^N(\widehat{\mathbf X}), \cP_{M_N}\tau \right\rangle
=
\left\langle \bsig^N(\widehat{\mathbf X}), \tau_N \right\rangle.
\]
Since \(\tau_N\in M_N\), by the definition of \(M_N\), there exist \(\widehat{\mathbf Y}^{1},\ldots,\widehat{\mathbf Y}^{M}\in\widehat{\cK}\) and scalars \\ \(\alpha_1,\ldots,\alpha_M\in\RR\) such that
\[
\tau_N
=
\sum_{i=1}^{M}
\alpha_i
\bsig^N(\widehat{\mathbf Y}^{i}).
\]

Hence, for every \(\widehat{\mathbf X}\in\widehat{\cK}\),
\[
\begin{aligned}
\left\langle
\bsig^N(\widehat{\mathbf X}), \tau
\right\rangle
&=
\left\langle
\bsig^N(\widehat{\mathbf X}), \tau_N
\right\rangle 
=
\left\langle
\bsig^N(\widehat{\mathbf X}), 
\sum_{i=1}^{M}
\alpha_i
\bsig^N(\widehat{\mathbf Y}^{i})
\right\rangle \\
&=
\sum_{i=1}^{M}
\alpha_i
\left\langle
\bsig^N(\widehat{\mathbf X}),
\bsig^N(\widehat{\mathbf Y}^{i})
\right\rangle 
=
\sum_{i=1}^{M}
\alpha_i
\bK_{\bsig}^{N}(\widehat{\mathbf X},\widehat{\mathbf Y}^{i}).
\end{aligned}
\]
Combining this identity with ~\eqref{bsig UAT} gives
\[
\sup_{\widehat{\bX} \in \widehat{\cK}}
\left|
f\left(\bsig^p(\widehat{\bX})\right)
-
\sum_{i=1}^{M}
\alpha_i
\bK^N_{\bsig}(\widehat{\bX},\widehat{\bY}^i)
\right|
<\varepsilon.
\]
This proves the claim.
\end{proof}

Instead of computing the branched signature kernel of the paths \(\bX\) and \(\bY\), we compute the geometric signature kernel of the extended paths \(\Bar{\bX}\) and \(\Bar{\bY}\) and leverage the Hairer--Kelly morphism \cite{hairer2015geometric} to show that the geometric signature kernel of extended paths gives the information that a branched signature kernel contains.

The Hairer--Kelly morphism \(\Psi\) is a Hopf algebra morphism from the Connes--Kreimer Hopf algebra of rooted trees \(\hopf\) to shuffle Hopf algebra over vector space \(\cV\) spanned by rooted trees i.e.,
\[
\Psi: (\hopf, \cdot, \Delta) \to (T(\mathcal{V}), \shuffle, \Bar{\Delta}).
\]
For a tree \(\tau \in \cF_n\), the morphism \(\Psi\) is defined as
\[
\Psi(\tau) = \tau + \Psi_{n-1}(\tau),
\]
where \(\Psi_{n-1}(\tau)\) is all the smaller trees coming from cutting the tree \(\tau\) in all possible ways.

The existence of such a morphism \(\Psi\) and the extended path \(\Bar{\bX}\) is proved in \cite{hairer2015geometric}. The main theorem in the paper states that for a branched rough path \(\bsig^N(\bX)\), one can construct the extended path \(\Bar{\bX}\) and the geometric rough path  \(\sig^N(\bX)\) such that
\[
\innerprod{\bsig^N(\bX)_{st}, \tau} = \innerprod{\sig^N(\Bar{\bX})_{st}, \Psi(\tau)},
\]
for all \(\tau \in \cF_n\). 

This result can also be applied to branched and geometric signatures using the extension theorems from rough paths to signatures. We state and prove the universal approximation theorem that claims that any function of the branched rough path can be approximated by the signature kernel of the extended paths. The precise statement and proof of the theorem are given as follows.

\begin{thm}
Let \(\cK\) be the compact subset of time-extended paths that admit a branched rough path lift i.e.,
\[
\cK \subset \left\{\bX: [0,T] \to \RR^{d+1} \text{ such that  }  \bX = (t,\bX^-) \text{ and  } \bX^- \text{ admits a branched rough path lift}\right\}
\]
Then for every continuous function \(f(\cK; \RR)\) and every \(\epsilon > 0\), there exist an integer \(N \ge 0\), paths \(\bY^1, \dots, \bY^M \in \cK\), and scalars \(\alpha_1, \dots \alpha_M\) such that
\[
\sup_{\bX \in \cK}
\left|
f\left(\bsig^p(\bX)\right)
-
\sum_{i=1}^{M}
\alpha_i
\bK^N_{\Psi, \sig}(\Bar{\bX},\Bar{\bY}^i)
\right|
<\varepsilon,
\]
where \(\bsig^p(\bX)\) is the branched rough path lift of \(\bX\) for some \(p \in \NN\) depending on the H\"older regularity of \(\bX^{-}\), \(\Bar{\bX},\Bar{\bY}^i\) are the extended paths constructed via extension map of Hairer-Kelly, and \(\bK^N_{\Psi, \sig}(\Bar{\bX},\Bar{\bY}^i)\) is defined as
\[
\bK^N_{\Psi, \sig}(\Bar{\bX},\Bar{\bY}^i) := \sum_{n=0}^{N} \sum_{\tau\in \cF_n} \left\langle                                                  \sig(\Bar{\bX}),\Psi(\tau)\right\rangle \left\langle                                                           \sig(\Bar{\bY^i}),\Psi(\tau)\right\rangle.
\]
\end{thm}

\begin{proof}
By the extension of Hairer-Kelly, for a path \(\bX\) that admits a branched rough path, there exists a map \(\Psi\) and an extended path \(\Bar{\bX}\) such that each component of the branched signature corresponding a tree \(\tau \in \hopf\) can be identified as the component of signature of \(\Bar{\bX}\) corresponding to \(\Psi(\tau)  \in T(\cV)\) i.e.,
\[
\innerprod{\bsig(\bX), \tau} = \innerprod{\sig(\Bar{\bX}), \Psi(\tau)}.
\]
This should also be true for truncated signatures. Therefore, for paths \(\bX\) and \(\bY_i\)
\[
\begin{aligned}
 \bK^N_{\bsig}(\bX,\bY^i) 
& = 
 \sum_{n=0}^{N} \sum_{\tau\in \cF_n} \left\langle                                                  \bsig(\bX),\tau\right\rangle \left\langle                                                           \bsig(\bY^i),\tau\right\rangle\\
& =  \sum_{n=0}^{N} \sum_{\tau\in \cF_n} \left\langle                                                  \sig(\Bar{\bX}),\Psi(\tau)\right\rangle \left\langle                                                           \sig(\Bar{\bY^i}),\Psi(\tau)\right\rangle
 =  \bK^N_{\Psi, \sig}(\Bar{\bX},\Bar{\bY}^i).
\end{aligned}
\]
Hence, the claim follows by Theorem~\ref{UAT_bsig_kernel}.
\end{proof}

\subsection{Path Lifts}
The universal approximation theorems above guarantee that any continuous functional of a branched rough path can be approximated via the signature kernel of an extended path $\Bar{\bX}$, but they do not prescribe how to construct the extension. We consider two approaches: an explicit \(t\)--value lift, and a learned neural network lift.
\begin{defn}[\(t\)--value lift]\label{def: tlift}
Suppose we observe a stream of forcing data $(t,F(t)) \in \mathbb R^2$. This stream of data can be lifted into $\mathbb R^3$ by augmenting the path to $(t,F(t),t^\alpha) \in \mathbb R^3$ where $\alpha \in (0,1)$.
\end{defn}
The motivation for the \(t\)--value lift comes from the case where the forcing \(F(t)\) is modeled by a fractional Brownian motion (fBm) \(\bB^H(t)\) with Hurst parameter \(H\in(0,\frac{1}{2})\). In this case, the covariance of fBm satisfies
\[
\mathbb{E}\left[\bB^H(t)\bB^H(t)\right]=t^{2H}.
\]
Thus, the deterministic component \(t^\alpha\) can be viewed as a way to encode the covariance related correction associated with the rough forcing. In particular, when \(\alpha=2H\), this lift agrees with the type of extended component appearing in the branched-to-geometric realization used in \cite{ali2025branchedsignaturemodel}.

In cases where the form of the driving signal is not known beforehand, one cannot prescribe an explicit path extension by hand. The path following a Brownian motion, fractional Brownian motion, or other rough frameworks would require a different extension framework. We therefore aim to construct an extension in a data dependent way. Neural Network path extensions were first considered in \cite{DeepSignatureTransforms} and noted as a way to increase the expressive power of the signature. The general method is mostly applied using convolution layers, recurrent neural networks, and sliding window architectures to lift the path before signatures are taken. In \cite{ali2025branchedsignaturemodel}, a neural network lift is applied using an MLP to construct a specific mapping that encodes the branched rough path information before taking signatures. We will apply this framework in our paper.

\begin{defn}[Neural Network Lift] \label{def: NN Lift}
Let $\bX_t: [0,T] \to \mathbb{R}^{d+1}$ be a time--extended path. We define a trainable map
$$
\Phi_{\theta}: \mathbb{R}^{d+1} \rightarrow \mathbb{R}^{m},
$$
with learnable parameters $\theta$. The learned coordinates may be defined as 
$$
\Bar \bX^{\theta}_t = \Phi_{\theta} (\bX_t) \in \mathbb{R}^m , \quad t \in [0,T].
$$
The lifted path is then constructed by concatenating the learned features with the original path
$$
\bar{\mathbb{X}}^{\theta}_t = (\bX_t, \Bar \bX^{\theta}_t) \in \mathbb{R}^{d+1+m} , \quad t \in [0,T].
$$
\end{defn}
The goal is to train for an extension which both improves model loss while staying consistent with the branched framework. We use a shuffle loss which enforces the integration by parts identity. The role of the shuffle loss is to make the learned coordinates \(\Bar{\bX}_\theta\) closer to geometric signature coordinates. At second order, this means that the product of two increments should agree with the sum of the two ordered iterated integrals, as in the integration-by-parts identity.

\begin{defn}[Shuffle Loss]
For \(a,b\in\{1,\ldots,m\}\) and \(t_i\in\pi[0,T]\), define
\[
D_i^a(\theta):=\Bar{\bX}_\theta^a(t_i)-\Bar{\bX}_\theta^a(t_0), \quad  \text{ and } \quad I_i^{a,b}(\theta)
:=
\int_{t_0}^{t_i}
\bigl(\Bar{\bX}_\theta^a(s)-\Bar{\bX}_\theta^a(t_0)\bigr)\,d\Bar{\bX}_\theta^b(s).
\]
The second-order shuffle residual is
\[
R_i^{a,b}(\theta)
:=
D_i^a(\theta)D_i^b(\theta)
-
I_i^{a,b}(\theta)
-
I_i^{b,a}(\theta).
\]
We then define the shuffle loss by
\[
\mathcal L_{\mathrm{shuffle}}(\theta)
:=
\frac{1}{N}
\sum_{t_i\in\pi[0,T]}
\sum_{a,b=1}^{m}
\left|R_i^{a,b}(\theta)\right|^2,
\]
where \(N=|\pi[0,T]|\) is the number of training time points.
\end{defn}
The integrals in \(I_i^{a,b}(\theta)\) are not computed through common signature libraries such as \textit{iisignature}, \textit{signatory} or \textit{keras\_sig}, since those libraries return geometric signatures and therefore satisfy the shuffle identities by construction. Instead, we use left-point Riemann sums, so that the non-geometric deviation can be measured directly. More precisely, if
\[
t_0=s_0<s_1<\cdots<s_{r_i}=t_i
\]
is a partition of \([t_0,t_i]\), then
\[
I_i^{a,b}(\theta)
:=
\sum_{\ell=1}^{r_i}
\bigl(\Bar{\bX}_\theta^a(s_{\ell-1})-\Bar{\bX}_\theta^a(s_0)\bigr)
\bigl(\Bar{\bX}_\theta^b(s_\ell)-\Bar{\bX}_\theta^b(s_{\ell-1})\bigr).
\]

\begin{defn}
The physics-informed loss is used to fit the observed path values along the training partition. For each \(t_i\in\pi[0,T]\), we first compute the truncated signature of the learned extended path \(\bbX_\theta\) over \([0,t_i]\), and then pass it through a predictor \(g_\phi\). The loss is defined by
\[
\mathcal{L}_{\mathrm{phy}}(\theta,\phi)
:=
\frac{1}{N}
\sum_{t_i\in\pi[0,T]}
\left\|
\bY_{t_i}
-
g_\phi\!\left(\sig^{k}(\bbX_\theta)_{0t_i}\right)
\right\|^2 .
\]
Here \(\pi[0,T]\) is a partition of the time interval, \(N=|\pi[0,T]|\) is the number of training time points, \(\sig^k(\bbX_\theta)_{0t_i}\) denotes the signature of the learned extended path truncated at level \(k\), and \(g_\phi\) is a trainable predictor, for example a linear layer or a small neural network.
\end{defn}
\begin{rem}
This is the most general form of the loss. For the problems of interest in the subsequent sections, we will define the physics--informed loss in different ways, depending on the nature of the problem. However, the shuffle loss will remain the same throughout.
\end{rem}
%%%%%%%%%%%%%%%%%%%%%%%%%%%%%%%%%%

\section{Branched Signature Kernel Solver for ODEs}\label{sec: kernel solvers}

In this section, we develop a branched signature-kernel solver based on a
single observed time series in $\R^{d}$---in the numerical experiments below,
this corresponds to the forcing term $f$ sampled at $N+1$ points on $[0,T]$.
From this single observation we generate $N+1$ training paths by the
count-sampling construction of Definition~\ref{defn:count-sampling}, which
serves as the foundation for the kernel-based solver developed in the
remainder of the section.

\subsection{Linear integral method I}\label{subsec:lin-inte}
First, let us define the linear ODE systems our method applies to. Fix the invertible matrices $A_0, A_1, \ldots, A_m \in \R^{d\times d}$, and define the linear operator $\cL$ acting on $C^{m}([0,T],\R^{d})$ by
\[
\cL \bu(t) \;:=\; \sum_{r=0}^{m} A_r\, \bu^{(r)}(t).
\]
The boundary operator $\cB$ collects the first $m$ derivatives at $t=0$:
\[
\cB \bu \;:=\; \bigl(\bu(0),\,\bu'(0),\,\ldots,\,\bu^{(m-1)}(0)\bigr)
\;=\; \bigl(\bg_0,\bg_1,\ldots,\bg_{m-1}\bigr) \;\in\; \R^{md},
\qquad \bg_p\in\R^{d}.
\]
\begin{defn}[Linear ODE system]\label{defn:lin-ode}
A Linear ODE system is defined as
\[
\cL\bu(t) \;=\; f(t), \quad\text{for}\quad t\in[0,T],\quad \text{and} \quad 
\cB\bu \;=\; \bg,
\]
for $\bu:[0,T]\to\R^{d}$, forcing $f:[0,T]\to\R^{d}$, and initial
data $\bg=(\bg_0,\ldots,\bg_{m-1})\in\R^{md}$.
\end{defn}

The scalar case $d=1$ recovers the single-equation form with constants
$k_r:=A_r\in\R$, and is our running example throughout the section.

Let $\bK^{N}_{\bsig}(\cdot,\cdot)$ be the branched signature kernel truncated at order $N$ defined in the previous section. We now construct an ODE solver based on this kernel (with the geometric signature kernel recovered as a special case) from a single trajectory observation of the forcing $f$. To our knowledge, this single-trajectory construction is new in the literature.

Let $0=t_0<t_1<\cdots<t_N=T$ be the observation times. By count-sampling
(Definition~\ref{defn:count-sampling}), the single observation of $f$
generates the $N+1$ training paths
\[
\bff_i \;:=\; \bigl(f(t_0),\,f(t_1),\,\ldots,\,f(t_i)\bigr)\;\in\;(\R^{d})^{i+1},
\qquad i=0,1,\ldots,N.
\]
For any query time $t\in[0,T]$ we write
$\bff(t):=\bigl(f(t_0),\ldots,f(t_{j(t)})\bigr)$ with
$j(t):=\max\{j:t_j\le t\}$, so that $\bff(t_j)=\bff_j$.

We place the component-wise kernel ansatz on the highest-order derivative $\bu^{(m)}$:
for each coordinate $q=1,\ldots,d$,
\begin{equation}\label{eq:ansatz-component}
u_q^{(m)}(t) \;=\; \sum_{i=0}^{N}\alpha_i^{(q)}\,
\bK^{N}_{\bsig}\!\bigl(\bff_i,\bff(t)\bigr),
\qquad q=1,\ldots,d.
\end{equation}
Stacking the coordinate coefficients into
$\bm\alpha_i:=\bigl(\alpha_i^{(1)},\ldots,\alpha_i^{(d)}\bigr)^{\top}\in\R^{d}$,
the ansatz reads in vector form
\begin{equation}\label{eq:ansatz-vector}
\bu^{(m)}(t) \;=\; \sum_{i=0}^{N}\bm\alpha_i\,
\bK^{N}_{\bsig}\!\bigl(\bff_i,\bff(t)\bigr).
\end{equation}

Successive integration of \eqref{eq:ansatz-vector} together with the
boundary data $\cB\bu=\bg$ gives, for $k=1,\ldots,m$,
\begin{equation}\label{eq:integration}
\bu^{(m-k)}(t)\;=\;\sum_{i=0}^{N}\bm\alpha_i\,
\bigl(\cI^{k}\bK^{N}_{\bsig}\bigr)\!\bigl(\bff_i,\bff(t)\bigr)
\;+\;\bm{p}_{k-1}(t),
\end{equation}
where $\cI^{k}$ is $k$-fold integration of the kernel in its second
argument,
\[
\bigl(\cI^{k}\bK^{N}_{\bsig}\bigr)\!\bigl(\bff_i,\bff(t)\bigr)
\;:=\;\int_{0}^{t}\!\!\int_{0}^{s_1}\!\!\cdots\!\int_{0}^{s_{k-1}}
\bK^{N}_{\bsig}\!\bigl(\bff_i,\bff(s_k)\bigr)\,
\d s_k\cdots\d s_1,
\]
and $\bm{p}_{k-1}(t)\in\R^{d}$ is the degree-$(k-1)$ Taylor polynomial
fixed by the prescribed initial conditions,
\[
\bm{p}_{k-1}(t)\;=\;\sum_{l=0}^{k-1}\frac{t^{l}}{l!}\,\bg_{\,m-k+l}.
\]
With the convention $\cI^{0}\bK^{N}_{\bsig}=\bK^{N}_{\bsig}$ and
$\bm{p}_{-1}\equiv 0$, formula \eqref{eq:integration} is also valid for
$k=0$ and recovers \eqref{eq:ansatz-vector}. In particular, $k=m$ yields
$\bu(t)$ itself.

Applying $\cL$ to the ansatz via \eqref{eq:integration} yields the collocation block system
\[
\cL\bu(t)\;=\;\sum_{i=0}^{N}\Bigl[\sum_{r=0}^{m}A_r\,
\bigl(\cI^{m-r}\bK^{N}_{\bsig}\bigr)\!\bigl(\bff_i,\bff(t)\bigr)\Bigr]\bm\alpha_i
\;+\;\sum_{r=0}^{m-1}A_r\,\bm{p}_{m-r-1}(t),
\]
and collocating $\cL\bu(t_j)=f(t_j)$ at every count-sampled path
$j=0,\ldots,N$ yields the block linear system
\begin{equation}\label{eq:collocation}
\sum_{i=0}^{N}L_{ji}\,\bm\alpha_i \;=\; \tilde f(t_j),
\qquad j=0,\ldots,N,
\end{equation}
with $d\times d$ blocks
\begin{equation}\label{eq:block-kernel}
L_{ji}\;:=\;\sum_{r=0}^{m}A_r\,
\bigl(\cI^{\,m-r}\bK^{N}_{\bsig}\bigr)\!\bigl(\bff_i,\bff_j\bigr)\;\in\;\R^{d\times d},
\end{equation}
and reduced forcing
$\tilde f(t_j):=f(t_j)-\sum_{r=0}^{m-1}A_r\,\bm{p}_{m-r-1}(t_j)\in\R^{d}$.
Equivalently, in matrix form,
\[
L\,\bm\alpha \;=\; \tilde F,
\qquad
L\in\R^{(N+1)d\times(N+1)d},\quad
\bm\alpha=\bigl(\bm\alpha_0^{\top},\ldots,\bm\alpha_N^{\top}\bigr)^{\top},\quad
\tilde F=\bigl(\tilde f(t_0)^{\top},\ldots,\tilde f(t_N)^{\top}\bigr)^{\top}.
\]

For our boundary constraints, evaluating the integrated ansatz~\eqref{eq:integration} at $t=0$ and using
$\bm{p}_{k-1}(0)=\bg_{m-k}$ gives, for $p=0,\ldots,m-1$,
\[
\bu^{(p)}(0)\;=\;\sum_{i=0}^{N}\bm\alpha_i\,
\bigl(\cI^{\,m-p}\bK^{N}_{\bsig}\bigr)\!\bigl(\bff_i,\bff_0\bigr)\Big|_{t=0}
\;+\;\bg_p\;=\;\bg_p,
\]
so that the iterated integrals vanish identically at $t=0$ and the initial
conditions are satisfied by construction through $\bm{p}_{k-1}$. The free
boundary constraints that remain are therefore the natural matching
relations at the degenerate first path $\bff_0=(f(t_0))$, encoded as the
block system
\begin{equation}\label{eq:boundary}
B\,\bm\alpha \;=\; \bg,
\qquad
B_{p,i}\;:=\;\bigl(\cI^{\,m-p}\bK^{N}_{\bsig}\bigr)\!\bigl(\bff_i,\bff_0\bigr)\,
\big|_{t=t_0}\in\R^{d\times d},
\quad p=0,\ldots,m-1,
\end{equation}
where the evaluation $|_{t=t_0}$ enforces compatibility at the boundary
node per Definition~\ref{defn:count-sampling}.

The combined collocation and boundary system form the following augmented least-squares system
\[
\begin{pmatrix} L \\ B \end{pmatrix}\bm\alpha
\;=\;
\begin{pmatrix} \tilde F \\ \bg \end{pmatrix}
\]
has $(N+1)d$ unknowns and $(N+1+m)d$ equations. We close it by the
weighted least-squares problem
\begin{equation}\label{eq:lsq}
\bm\alpha
\;=\;\arg\min_{\bm\alpha\in\R^{(N+1)d}}\,
\bigl\lVert L\,\bm\alpha-\tilde F\bigr\rVert^{2}
\;+\;\lambda_{B}\,\bigl\lVert B\,\bm\alpha-\bg\bigr\rVert^{2},
\end{equation}
with tunable boundary weight $\lambda_{B}>0$.

Once $\bm\alpha$ has been recovered from \eqref{eq:lsq}, the solution evaluated at a new query time $\tau\in[0,T]$ is obtained by integrating the ansatz $m$ times,
\begin{equation}\label{eq:prediction}
\bu(\tau)\;=\;\sum_{i=0}^{N}\bm\alpha_i\,
\bigl(\cI^{m}\bK^{N}_{\bsig}\bigr)\!\bigl(\bff_i,\bff(\tau)\bigr)
\;+\;\bm{p}_{m-1}(\tau),
\end{equation}
where the same count-sampled training paths $\bff_0,\ldots,\bff_N$ act as
the kernel anchors.

\begin{example}\label{ex:scalar-integral}
Consider the scalar second-order ODE
\begin{align*}
&k_2\,u''(t) + k_1\,u'(t) + k_0\,u(t) \;=\; f(t),\qquad t\in[0,T],\\
&u(0) \;=\; a,\qquad u'(0) \;=\; b,
\end{align*}
with constants $k_0,k_1,k_2\in\R$ ($k_2\neq 0$), corresponding to the
$d=1$, $m=2$ instance of Definition~\ref{defn:lin-ode} with
$A_r=k_r$ and initial data $g_0=a$, $g_1=b$.
\end{example}

We use Example \ref{ex:scalar-integral} to demonstrate our general solver construction. Following the integral method of Section~\ref{subsec:lin-inte}, we
place the kernel ansatz on the highest-order derivative $u''$ and
recover the lower derivatives and the solution itself by integrating the
kernel (not the forcing $f$):
\begin{equation}\label{eq:ex-ansatz-highest}
u''(t) \;=\; \sum_{i=0}^{N}\alpha_i\,\bK^{N}_{\bsig}\!\bigl(\bff_i,\bff(t)\bigr),
\end{equation}
where $\bff_0,\ldots,\bff_N$ are the count-sampled training paths
$\bff_i=(f(t_0),\ldots,f(t_i))$ on the grid $0=t_0<t_1<\cdots<t_N=T$.

Following the method derived in Section \ref{subsec:lin-inte}, and the initial condition 
$u'(0)=b$, $u(0)=a$, we get 
\begin{align*}
u'(t) \;=\; b + \int_{0}^{t} u''(s)\,\dd s
\;=\; b + \sum_{i=0}^{N}\alpha_i\,\bigl(\cI\bK^{N}_{\bsig}\bigr)\!\bigl(\bff_i,\bff(t)\bigr);\\
u(t) \;=\; a + b\,t + \int_{0}^{t}\!\!\int_{0}^{s} u''(r)\,\dd r\,\dd s
\;=\; a + b\,t + \sum_{i=0}^{N}\alpha_i\,\bigl(\cI^{2}\bK^{N}_{\bsig}\bigr)\!\bigl(\bff_i,\bff(t)\bigr).
\end{align*}
which then constructs the equation for equation~\eqref{eq:integration}
with boundary polynomials $\bm{p}_0(t)=b$ and $\bm{p}_1(t)=a+b\,t$. 

Collocating the ODE at $t=t_j$ for $j=0,\ldots,N$ and substituting the
three expressions above gives the matrix linear system
\[
\sum_{i=0}^{N}\alpha_i\Bigl[
k_2\bK^{N}_{\bsig}\!\bigl(\bff_i,\bff_j\bigr)
+k_1\bigl(\cI\bK^{N}_{\bsig}\bigr)\!\bigl(\bff_i,\bff_j\bigr)
+k_0\bigl(\cI^{2}\bK^{N}_{\bsig}\bigr)\!\bigl(\bff_i,\bff_j\bigr)
\Bigr]\!=\!f(t_j) - k_1b - k_0\bigl(a + bt_j\bigr).
\]
Introduce the three $(N+1)\times(N+1)$ Gram matrices
\[
K_{ji} \;:=\; \bK^{N}_{\bsig}\!\bigl(\bff_i,\bff_j\bigr),\qquad
K^{(1)}_{ji} \;:=\; \bigl(\cI\bK^{N}_{\bsig}\bigr)\!\bigl(\bff_i,\bff_j\bigr),\qquad
K^{(2)}_{ji} \;:=\; \bigl(\cI^{2}\bK^{N}_{\bsig}\bigr)\!\bigl(\bff_i,\bff_j\bigr),
\]
so that the ansatz values at the collocation nodes are
\[
\bu''(\bm t) \;=\; K\,\bm\alpha,\qquad
\bu'(\bm t) \;=\; b\,\mathbf{1} + K^{(1)}\bm\alpha,\qquad
\bu(\bm t) \;=\; a\,\mathbf{1} + b\,\bm t + K^{(2)}\bm\alpha,
\]
with $\bm t=(t_0,\ldots,t_N)^{\top}$,
$\bm\alpha=(\alpha_0,\ldots,\alpha_N)^{\top}$, and $\mathbf{1}\in\R^{N+1}$ the
all-ones vector. Substituting into the ODE evaluated at $\bm t$,
\[
k_2\,K\bm\alpha + k_1\bigl(b\,\mathbf{1}+K^{(1)}\bm\alpha\bigr) + k_0\bigl(a\,\mathbf{1}+b\,\bm t+K^{(2)}\bm\alpha\bigr) \;=\; \bm f,
\qquad \bm f := \bigl(f(t_0),\ldots,f(t_N)\bigr)^{\top},
\]
collecting terms yields the block linear system
\begin{equation}\label{eq:ex-linsys}
L\,\bm\alpha \;=\; \tilde{F},
\qquad
L \;:=\; k_2\,K + k_1\,K^{(1)} + k_0\,K^{(2)}\;\in\;\R^{(N+1)\times(N+1)},
\end{equation}
with reduced forcing
\[
\tilde{F} \;:=\; \bm f \;-\; \bigl(k_0\,a + k_1\,b\bigr)\,\mathbf{1} \;-\; k_0\,b\,\bm t \;\in\;\R^{N+1}.
\]
Equation~\eqref{eq:ex-linsys} is the scalar $m=2$ specialization of the
block kernel system~\eqref{eq:block-kernel}, with no integration applied
to $f$. We demonstrate this example in Section \ref{sec:numerical-experiments} with a real-world application.

\subsection{Linear integral method II}\label{subsec:lin-inte-2}
In Section~\ref{subsec:lin-inte}, we place the kernel ansatz on the highest-order derivative $\bu^{(m)}$ and integrate $m$ times to recover $\bu$. An equivalent formulation is to integrate \emph{both sides} of the ODE $m$ times first, turning the differential equation into an integral equation in $\bu$ itself, and then place the kernel ansatz directly on $\bu$. This method is easier to implement, but less robust than the first, as seen in the examples section. We begin with a second-order scalar ODE to demonstrate the main idea. To fix notation, consider a damped oscillator with prescribed parameters $(m,c,k)$ :
\[
m\,u''(t) + c\,u'(t) + k\,u(t) = f(t),\qquad
u(t_0)=u_0,\quad u'(t_0)=v_0.
\]
Here $f(t)$ denotes the external forcing signal.
Integrating both sides twice from $t_0$ to $t$ and applying the fundamental
theorem of calculus together with the initial conditions yields the
Volterra integral equation
\begin{equation}\label{eq:scalar-volterra}
m\,u(t) \;+\; c\!\int_{t_0}^{t}\! u(s)\,\dd s
\;+\; k\!\int_{t_0}^{t}\!\!\int_{t_0}^{s}\! u(r)\,\dd r\,\dd s
\;=\; \int_{t_0}^{t}\!\!\int_{t_0}^{s}\! f(r)\,\dd r\,\dd s
\;+\; m u_0 \;+\; (m v_0 + c u_0)(t-t_0).
\end{equation}
The differential equation has been transformed into an integral equation
in $u$ alone, whose right-hand side depends only on the iterated integral
of $f$ and on the initial data. No derivatives of $u$ remain on the
left.

The method may also be extended to a general order-$m$, $d$-dimensional system using the following. Apply the $m$-fold integration operator
\[
\bigl(\cI_{t_0}^{m}\phi\bigr)(t)\;:=\;
\int_{t_0}^{t}\!\!\int_{t_0}^{s_1}\!\!\cdots\!\int_{t_0}^{s_{m-1}}\!\!
\phi(s_m)\,\dd s_m\cdots\dd s_1
\]
to both sides of $\cL\bu(t)=f(t)$. Repeated use of the fundamental theorem
of calculus, together with the initial conditions
$\bu^{(l)}(t_0)=\bg_l$ for $l=0,\ldots,m-1$, gives the identity
\begin{equation}\label{eq:taylor-remainder}
\bigl(\cI_{t_0}^{m}\bu^{(r)}\bigr)(t)
\;=\;\bigl(\cI_{t_0}^{\,m-r}\bu\bigr)(t)
\;-\;\sum_{l=0}^{r-1}\frac{(t-t_0)^{m-r+l}}{(m-r+l)!}\,\bg_l,
\qquad r=0,1,\ldots,m,
\end{equation}
which for $r=m$ reduces to Taylor's theorem with integral remainder and
for $r=0$ is the trivial identity $\cI^{m}\bu=\cI^{m}\bu$. Summing
\eqref{eq:taylor-remainder} against $A_r$ and rearranging,
\begin{equation}\label{eq:volterra}
\sum_{r=0}^{m} A_r\,\bigl(\cI_{t_0}^{\,m-r}\bu\bigr)(t)
\;=\;\bigl(\cI_{t_0}^{m}f\bigr)(t)\;+\;\bm{q}(t),
\end{equation}
where the boundary polynomial $\bm{q}(t)\in\R^{d}$ collects all
initial-conditions,
\begin{equation}\label{eq:q-poly}
\bm{q}(t)\;:=\;\sum_{r=1}^{m} A_r\sum_{l=0}^{r-1}
\frac{(t-t_0)^{m-r+l}}{(m-r+l)!}\,\bg_l.
\end{equation}
Equation~\eqref{eq:volterra} is the multi-dimensional, order-$m$ Volterra
form of Definition~\ref{defn:lin-ode}: all initial conditions
$\bg_0,\ldots,\bg_{m-1}$ are absorbed into $\bm{q}$, and only iterated
integrals of $f$ and $\bu$ appear.
\begin{rem} For $m=2$, $d=1$, with
$A_0=k$, $A_1=c$, $A_2=m$, formula~\eqref{eq:q-poly} gives
$\bm{q}(t)=m u_0 + (m v_0 + c u_0)(t-t_0)$, recovering
\eqref{eq:scalar-volterra}.
\end{rem}

Because~\eqref{eq:volterra} contains no derivatives of $\bu$, we place
the kernel ansatz directly on the solution $\bu$: for each coordinate
$q=1,\ldots,d$,
\begin{equation}\label{eq:ansatz-component-II}
u_q(t) \;=\; \sum_{i=0}^{N}\alpha_i^{(q)}\,
\bK^{N}_{\bsig}\!\bigl(\bff_i,\bff(t)\bigr),
\end{equation}
or, in vector form with
$\bm\alpha_i:=\bigl(\alpha_i^{(1)},\ldots,\alpha_i^{(d)}\bigr)^{\top}\in\R^{d}$,
\begin{equation}\label{eq:ansatz-vector-II}
\bu(t) \;=\; \sum_{i=0}^{N}\bm\alpha_i\,
\bK^{N}_{\bsig}\!\bigl(\bff_i,\bff(t)\bigr).
\end{equation}
The count-sampled training paths $\bff_0,\ldots,\bff_N$ defined in
Section~\ref{subsec:lin-inte} are reused as the kernel anchors. Substituting~\eqref{eq:ansatz-vector-II} into~\eqref{eq:volterra} and
collocating at $t=t_j$ for $j=0,\ldots,N$ gives
\begin{equation}\label{eq:collocation-II}
\sum_{i=0}^{N} L_{ji}\,\bm\alpha_i \;=\; \tilde{f}_{\mathrm{II}}(t_j),
\qquad j=0,\ldots,N,
\end{equation}
with the \emph{same} $d\times d$ block kernel matrix as in the original method~\eqref{eq:block-kernel},
\begin{equation}\label{eq:block-kernel-II}
L_{ji}\;=\;\sum_{r=0}^{m}A_r\,
\bigl(\cI_{t_0}^{\,m-r}\bK^{N}_{\bsig}\bigr)\!\bigl(\bff_i,\bff_j\bigr)
\;\in\;\R^{d\times d},
\end{equation}
and integrated forcing
\[
\tilde{f}_{\mathrm{II}}(t_j)\;:=\;
\bigl(\cI_{t_0}^{m}f\bigr)(t_j)\;+\;\bm{q}(t_j)\;\in\;\R^{d}.
\]
In the Matrix form, we then have
\begin{equation}
L\,\bm\alpha = \tilde{F}_{\mathrm{II}},\quad \text{with}\quad 
\tilde{F}_{\mathrm{II}}=\bigl(\tilde{f}_{\mathrm{II}}(t_0)^{\top},\ldots,
\tilde{f}_{\mathrm{II}}(t_N)^{\top}\bigr)^{\top}\in\R^{(N+1)d}.
\end{equation}

% {\color{red}
% \textbf{Evaluation within $[0,T]$.}
% Once $\bm\alpha$ has been recovered from~\eqref{eq:collocation-II}, the
% solution at a new query time $\tau\in[0,T]$ is read off from the ansatz
% directly:
% \begin{equation}\label{eq:prediction-II}
% \bu(\tau)\;=\;\sum_{i=0}^{N}\bm\alpha_i\,
% \bK^{N}_{\bsig}\!\bigl(\bff_i,\bff(\tau)\bigr),
% \end{equation}
% without any additional integration --- in contrast to the differential
% method~\eqref{eq:prediction}, which requires an $m$-fold kernel integral.}

%\subsection{Nonlinear equations and the integral method}\label{subsec:nonlin-eqs}

\subsection{Nonlinear integral method}\label{subsec:nonlin-inte}

We extend the framework of Section~\ref{subsec:lin-inte} to nonlinear
systems. Let $\cN$ denote a (possibly nonlinear) $m$-th order differential
operator acting on $C^{m}([0,T],\R^{d})$, and recall the boundary
operator $\cB\bu:=(\bu(0),\bu'(0),\ldots,\bu^{(m-1)}(0))$.

\begin{defn}[Nonlinear ODE system]\label{defn:nonlin-ode}
A Nonlinear ODE system is defined as
\[
\cN\bu(t) \;=\; f(t), \quad \text{for}\quad t\in[0,T],\quad \text{and}\quad
\cB\bu \;=\; \bg,
\]
for unknown $\bu:[0,T]\to\R^{d}$, forcing $f:[0,T]\to\R^{d}$, and initial
data $\bg=(\bg_0,\ldots,\bg_{m-1})\in\R^{md}$. A typical splitting is
$\cN\bu=\cL\bu+\cN_{\mathrm{nl}}(\bu,\bu',\ldots,\bu^{(m-1)})$, where
$\cL$ is the linear operator of Definition~\ref{defn:lin-ode} and
$\cN_{\mathrm{nl}}$ collects the nonlinear contributions.
\end{defn}

We illustrate the method with the Duffing equation, which serves as one of our main numerical examples, before presenting the general multidimensional formulation.

\begin{defn}[Duffing equation]\label{defn:duffing}
\begin{align*}
&u''(t) + k_1\,u'(t) + k_0\,u(t) + \gamma\,u(t)^{3} \;=\; f(t),\qquad t\in[0,T],\\
&u(0)=a,\qquad u'(0)=b,
\end{align*}
with linear constants $k_0,k_1\in\R$, nonlinear stiffness $\gamma\in\R$,
and initial data $a,b\in\R$. This is the $d=1$, $m=2$ instance of
Definition~\ref{defn:nonlin-ode} with linear part
$\cL u=u''+k_1 u'+k_0 u$ (so $A_2=1$, $A_1=k_1$, $A_0=k_0$) and
nonlinear part $\cN_{\mathrm{nl}}(u)=\gamma\,u^{3}$.
\end{defn}

Following the integral method of Section~\ref{subsec:lin-inte-2}, place
the kernel ansatz on $u''$ and recover the lower derivatives by
integrating the kernel:
\begin{align}
u''(t) &\;=\; \sum_{i=0}^{N}\alpha_i\,\bK^{N}_{\bsig}\!\bigl(\bff_i,\bff(t)\bigr),
\label{eq:duff-ansatz-u2}\\[2pt]
u'(t)  &\;=\; b \;+\; \sum_{i=0}^{N}\alpha_i\,\bigl(\cI\bK^{N}_{\bsig}\bigr)\!\bigl(\bff_i,\bff(t)\bigr),
\label{eq:duff-ansatz-u1}\\[2pt]
u(t)   &\;=\; a + b\,t \;+\; \sum_{i=0}^{N}\alpha_i\,\bigl(\cI^{2}\bK^{N}_{\bsig}\bigr)\!\bigl(\bff_i,\bff(t)\bigr),
\label{eq:duff-ansatz-u0}
\end{align}
where $\bff_0,\ldots,\bff_N$ are the count-sampled training paths on the
grid $0=t_0<t_1<\cdots<t_N=T$ from Section~\ref{subsec:lin-inte}, and the
boundary polynomials $b$ and $a+b\,t$ are the cases $k=1,2$ of
\eqref{eq:integration}.
Plugging \eqref{eq:duff-ansatz-u2}--\eqref{eq:duff-ansatz-u0} into
Definition~\ref{defn:duffing} and grouping the linear and nonlinear
contributions,
\begin{align}\label{example:non-linear}
\underbrace{\sum_{i=0}^{N}\alpha_i\Bigl[\bK^{N}_{\bsig}\!\bigl(\bff_i,\bff(t)\bigr)
+ k_1\bigl(\cI\bK^{N}_{\bsig}\bigr)\!\bigl(\bff_i,\bff(t)\bigr)
+ k_0\bigl(\cI^{2}\bK^{N}_{\bsig}\bigr)\!\bigl(\bff_i,\bff(t)\bigr)\Bigr]}_{\text{linear part}}
\;\nonumber \\
+\;
\underbrace{\gamma\Bigl[a + b\,t + \sum_{i=0}^{N}\alpha_i\bigl(\cI^{2}\bK^{N}_{\bsig}\bigr)\!\bigl(\bff_i,\bff(t)\bigr)\Bigr]^{3}}_{\text{nonlinear part}}=\; f(t) \;-\; k_1\,b \;-\; k_0\bigl(a+b\,t\bigr).
\end{align}

The boundary-data contributions $k_1 b$ and $k_0(a+b\,t)$ have been moved
to the right-hand side, exactly as in the linear case
(equation~\eqref{eq:ex-linsys}).
Given the above derivation, we then
introduce the three Gram matrices, the time vector, and the unit
vector,
\begin{align*}
K_{ji} &:= \bK^{N}_{\bsig}\!\bigl(\bff_i,\bff_j\bigr),\quad
K^{(1)}_{ji} := \bigl(\cI\bK^{N}_{\bsig}\bigr)\!\bigl(\bff_i,\bff_j\bigr),\quad\\
K^{(2)}_{ji} &:= \bigl(\cI^{2}\bK^{N}_{\bsig}\bigr)\!\bigl(\bff_i,\bff_j\bigr),
\quad
\bm t = (t_0,\ldots,t_N)^{\top},\quad
\mathbf{1}\in\R^{N+1},
\end{align*}

so that, at $t=\bm t$,
\[
u''(\bm t) = K\bm\alpha,\qquad
u'(\bm t)  = b\,\mathbf{1} + K^{(1)}\bm\alpha,\qquad
u(\bm t)   = a\,\mathbf{1} + b\,\bm t + K^{(2)}\bm\alpha.
\]
Substituting into the Duffing equation evaluated at $\bm t$ gives the
nonlinear residual equation
\begin{equation}\label{eq:duff-residual}
L\,\bm\alpha \;+\; \gamma\,\bigl(a\,\mathbf{1} + b\,\bm t + K^{(2)}\bm\alpha\bigr)^{\circ 3}
\;=\; \tilde F,
\end{equation}
with linear block matrix
\[
L \;:=\; K + k_1\,K^{(1)} + k_0\,K^{(2)} \;\in\;\R^{(N+1)\times(N+1)},
\]
reduced forcing
\[
\tilde F \;:=\; \bm f \;-\; \bigl(k_0\,a + k_1\,b\bigr)\,\mathbf{1} \;-\; k_0\,b\,\bm t
\;\in\;\R^{N+1},\qquad
\bm f = \bigl(f(t_0),\ldots,f(t_N)\bigr)^{\top},
\]
and $(\,\cdot\,)^{\circ 3}$ denoting the elementwise  cube of a
vector. Equation~\eqref{eq:duff-residual} reduces to the linear block
system~\eqref{eq:ex-linsys} when $\gamma=0$.

For our nonlinear solver, we define the residual map and the training loss
\begin{align}
&\bm{\cR}(\bm\alpha)
:= L\,\bm\alpha \;+\; \gamma\,\bigl(a\,\mathbf{1} + b\,\bm t + K^{(2)}\bm\alpha\bigr)^{\circ 3} \;-\; \tilde F, \\
&\mathrm{Loss}_{\mathrm{ODE}}(\bm\alpha)
:= \frac{1}{N+1}\,\bigl\lVert \bm{\cR}(\bm\alpha)\bigr\rVert_{2}^{2}.
\end{align}
Because $\bm{\cR}$ is a smooth (cubic) polynomial in $\bm\alpha$, we
recover the optimal weights via a quasi-Newton scheme:
\[
\bm\alpha^{\star} \;=\;\operatorname*{argmin}_{\bm\alpha\in\R^{N+1}}\,\mathrm{Loss}_{\mathrm{ODE}}(\bm\alpha),
\]
solved with L-BFGS. Tikhonov
regularization $+\lambda\,\lVert\bm\alpha\rVert_{2}^{2}$ may be added to
the loss for stability.

Once $\bm\alpha^{\star}$ has been recovered, the solution and its
derivatives can be evaluated at a new query time $\tau\in[0,T]$, read off from the
ansatz~\eqref{eq:duff-ansatz-u2}--\eqref{eq:duff-ansatz-u0} directly,

\begin{equation}\label{eq:nonlinear-prediction}
\begin{gathered}
u''(\tau) = \sum_{i=0}^{N}\alpha_i^{\star}\,\bK^{N}_{\bsig}\!\bigl(\bff_i,\bff(\tau)\bigr),\qquad
u'(\tau)  = b + \sum_{i=0}^{N}\alpha_i^{\star}\,\bigl(\cI\bK^{N}_{\bsig}\bigr)\!\bigl(\bff_i,\bff(\tau)\bigr),
\\[0.5em]
u(\tau)   \;=\; a + b\,\tau \;+\; \sum_{i=0}^{N}\alpha_i^{\star}\,\bigl(\cI^{2}\bK^{N}_{\bsig}\bigr)\!\bigl(\bff_i,\bff(\tau)\bigr).
\end{gathered}
\end{equation}

The method may also be applied in the multi-dimensional case. For a general operator $\cN\bu=\cL\bu+\cN_{\mathrm{nl}}(\bu,\bu',\ldots,\bu^{(m-1)})$ with $\bu:[0,T]\to\R^{d}$, the same recipe yields the multi-dimensional nonlinear residual
\[
\bm{\cR}(\bm\alpha)\;=\;L\,\bm\alpha\;+\;\bm{\Phi}_{\mathrm{nl}}\bigl(K^{(0)}\bm\alpha,K^{(1)}\bm\alpha,\ldots,K^{(m-1)}\bm\alpha;\,\bm{p}_{m-1},\ldots,\bm{p}_{0}\bigr)\;-\;\tilde F,
\]
where $L$ is the linear block-kernel matrix of \eqref{eq:block-kernel},
$\bm{\Phi}_{\mathrm{nl}}$ encodes the nodewise evaluation of
$\cN_{\mathrm{nl}}$ at the integrated-kernel reconstructions of
$\bu,\bu',\ldots,\bu^{(m-1)}$, and the Taylor polynomials $\bm{p}_{k-1}$
account for the initial data. The minimization is again carried out by
L-BFGS in the $(N+1)d$ unknowns $\bm\alpha=(\bm\alpha_0^{\top},\ldots,\bm\alpha_N^{\top})^{\top}$.

%%%%%%%%%%%%%%%%%%%%%%%%%%%%%%%%%%
\section{Test/Train and Retrain Method}\label{sec:test-train}

In this section, we build on the branched signature kernel solvers of Section~\ref{sec: kernel solvers}, detailing training, prediction, and retraining algorithms for both linear and nonlinear systems of ODEs. We also describe the pipeline for the neural network extension models and discuss the training process. 
Within the models, one may choose to either employ the linear signature kernel $\bK^{N}_{\bsig}(\bff_i,\bff_j)$ defined in Definition~\ref{def:linsigkernel}, or the rbf signature kernel $\bK^{N,\mathrm{RBF}}_{\bsig}(\bff_i,\bff_j)$, as defined in Definition~\ref{def:rbfsigkernel}; the latter being prefered for particularly rough or nonlinear problems. 
Our Signature kernel construction is based on the count method, where the count sampling paths have different length, thus the RBF kernel is a type of normalization in our framework.

\subsection{Linear training and prediction}

We first describe the linear ODE system prediction and retraining methods for a single forcing trajectory $f(t)$, $t\ge 0$.
% \subsection{Initial training}\label{subsec:tt-init}
Let $\bff_0,\ldots,\bff_n$ be the count-sampled training paths from the first $n+1$ observations, $\bff_i=(f(t_0),\ldots,f(t_i))$, and let $\bm t^{(n)}=(t_0,\ldots,t_n)^{\top}$, $\bm f^{(n)}=\bigl(f(t_0)^{\top},\ldots,f(t_n)^{\top}\bigr)^{\top}\in\R^{(n+1)d}$. For the initial batch of $n{+}1$ observations we solve
\[
L_n\,\bm\alpha^{(n)} = \tilde F^{(n)},
\]
with $L_n$, $\tilde F^{(n)}$, $K^{(k)}$ and $\bm P_k^{(n)}$ as in
Section~\ref{subsec:lin-inte}, using either a direct solver or
Tikhonov-regularized least squares.

% \subsection{Standard Prediction}\label{subsec:tt-pred}

In the above training process, we fit a single batch of coefficients $\bm\alpha\in\mathbb{R}^{N+1}$ on a fixed observation window $\{t_0,\ldots,t_N\}$. The standard prediction method for the linear cases uses ~\eqref{eq:prediction} to predict on new points, and periodically retrains using the training method above.

The standard training and prediction method is extremely fast, but may suffer from quickly degrading performance if parameters are not tuned correctly or the retraining interval is not frequent enough. One may instead opt to use a rolling retraining algorithm, trading implementation complexity/speed for better retraining frequency and less hyperparameter tuning. 

\subsubsection{Rolling Prediction}\label{subsec:tt-online}

In rolling-retrain prediction we update the previous coefficient vector by adding
one new $d$-dimensional block per incoming observation, using a closed-form update in the linear case, and a scalar Newton step in the nonlinear case.

When the next observation $f(t_{n+1})\in\R^{d}$ arrives, the collocation
system grows to size $(n+2)d\times(n+2)d$. Let $L_{n+1}$ and
$\tilde F^{(n+1)}$ be the augmented block matrix and forcing as in
Section~\ref{subsec:lin-inte}. Instead of refitting all $n{+}2$ blocks, we
freeze $\bm\alpha_0,\ldots,\bm\alpha_n$ and solve only for
$\bm\alpha_{n+1}\in\R^{d}$ from the new block row
\[
\sum_{k=0}^{n+1}\Phi_{n+1,k}\,\bm\alpha_k = \tilde f^{(n+1)}(t_{n+1}),
\]
where
\[
\Phi_{j,i} := \sum_{r=0}^{m}A_r\,\bigl(\cI^{\,m-r}\bK^{N}_{\bsig}\bigr)(\bff_i,\bff_j),
\qquad
\tilde f^{(n+1)}(t_{n+1}) := f(t_{n+1}) - \sum_{r=0}^{m-1}A_r\,\bm{p}_{m-r-1}(t_{n+1}),
\]

Isolating $\bm\alpha_{n+1}$ gives the closed-form update
\begin{equation}\label{eq:tt-update}
\hat{\bm\alpha}_{n+1}
\;=\;
\Phi_{n+1,n+1}^{-1}\Bigl[\,\tilde f^{(n+1)}(t_{n+1}) \;-\; \sum_{k=0}^{n}\Phi_{n+1,k}\,\bm\alpha_k\,\Bigr].
\end{equation}
The diagonal block
$\Phi_{n+1,n+1}=\sum_{r=0}^{m}A_r\,(\cI^{\,m-r}\bK^{N}_{\bsig})(\bff_{n+1},\bff_{n+1})$
is invertible whenever $A_m$ is invertible and the time step is positive
(the iterated integrals $\cI^{\,k}\bK^{N}_{\bsig}$ contribute scalar
weights that decay with $k$, so the leading $A_m$ term dominates for non-degenerate $t_{n+1}>t_n$). The solution at the new node follows from the prediction formula~\eqref{eq:prediction},
\[
\bu(t_{n+1})
= \bm{p}_{m-1}(t_{n+1})
+ \sum_{k=0}^{n}\bm\alpha_k\,\bigl(\cI^{m}\bK^{N}_{\bsig}\bigr)(\bff_k,\bff_{n+1})
+ \hat{\bm\alpha}_{n+1}\,\bigl(\cI^{m}\bK^{N}_{\bsig}\bigr)(\bff_{n+1},\bff_{n+1}).
\]

By induction, \eqref{eq:tt-update} produces $\hat{\bm\alpha}_{n+2},\hat{\bm\alpha}_{n+3},\ldots$
using only previously fitted weights. Each online step costs
$O\bigl(n\,d^{2}\bigr)$ to assemble the new kernel block-row plus
$O(d^{3})$ to invert $\Phi_{n+1,n+1}$, versus $O\bigl((nd)^{3}\bigr)$ for a full refit. The frozen-coefficient update is exact in $\bm\alpha_{n+1}$ given $\bm\alpha_0,\ldots,\bm\alpha_n$, but not equal to the full batch solution on $n{+}2$ samples. As a result, the accumulated approximation error motivates the retraining schedule introduced below.

As prediction accuracy degrades or drift becomes apparent, periodic retrains are required. These retrains refit the entire set $\alpha$ over the new interval using the same method as initial training.

%\subsection{Periodic retraining}\label{subsec:tt-retrain}

\subsection{Non-linear training and prediction}\label{subsec:tt-init-nonlin}

For the nonlinear ODE system we define in Definition~\ref{defn:nonlin-ode}, the batch collocation residual from Section~\ref{subsec:nonlin-inte} with
operator $\cN\bu=\cL\bu+\cN_{\mathrm{nl}}(\bu,\bu',\ldots,\bu^{(m-1)})$ is
\[
\bm{\cR}\bigl(\bm\alpha^{(n)}\bigr)
\;=\; L_n\,\bm\alpha^{(n)} \;+\; \bm{\Phi}_{\mathrm{nl}}\bigl(\bm\alpha^{(n)}\bigr) \;-\; \tilde F^{(n)},
\]

where $\bm{\Phi}_{\mathrm{nl}}(\bm\alpha^{(n)})\in\R^{(n+1)d}$ stacks the
componentwise contributions $\cN_{\mathrm{nl}}\bigl(\bu(t_j),  \bu'(t_j),\\ \ldots,\bu^{(m-1)}(t_j)\bigr)$ evaluated via the integrated-kernel reconstructions for $r=0,\ldots,m-1$: $\bu^{(r)}(t_j)=\bm{p}_{m-r-1}(t_j)+\sum_{k=0}^{n}\bm\alpha_k\,(\cI^{\,m-r}\bK^{N}_{\bsig})(\bff_k,\bff_j)$. Initial training minimizes
$\|\bm{\cR}\|_2^{2}+\lambda\|\bm\alpha\|_2^{2}$ by L-BFGS.

%\subsection{Standard Predict and Retrain (Nonlinear)}\label{subsec:tt-pred-nonlin}
In the nonlinear case, prediction is done using ~\eqref{eq:nonlinear-prediction}. However, when periodically retraining, instead of reapplying the full training from Section~\ref{subsec:nonlin-inte} a speedup is implemented using previous training values for $\alpha$ to provide a warm start.
When new observations $\{f_{n+1},\ldots,f_{n+B}\}$ arrive, we warm-start L-BFGS from the padded vector
\[
\bm\alpha^{(n+B)}_{\mathrm{init}}
=
\bigl(\bm\alpha^{(n)\top},\,\mathbf{0}_{Bd}^\top\bigr)^\top
\in \mathbb{R}^{(n+B+1)d},
\]
and minimise $\mathcal{L}_{n+B}$ to obtain $\bm\alpha^{(n+B)}$,
which is then used in~\eqref{eq:nonlinear-prediction}.

\subsubsection{Rolling prediction (nonlinear)}\label{subsec:tt-rolling-nonlin}

For the nonlinear rolling update, when $f(t_{n+1})$ arrives we freeze
$\bm\alpha_0,\ldots,\bm\alpha_n$ and solve only for the new block
$\bm\alpha_{n+1}\in\R^{d}$. Define the previous-coefficient
contributions at the new node
\[
\bu_{\mathrm{prev}}^{(r)}(t_{n+1})
:= \bm{p}_{m-r-1}(t_{n+1})
+ \sum_{k=0}^{n}\bm\alpha_k\,\bigl(\cI^{\,m-r}\bK^{N}_{\bsig}\bigr)(\bff_k,\bff_{n+1}),
\]
for $r=0,1,\ldots,m$, and the following diagonal self-kernel scalars so that
\[
\bu^{(r)}(t_{n+1})
= \bu_{\mathrm{prev}}^{(r)}(t_{n+1}) + \kappa_r\,\bm\alpha_{n+1}.
\quad
\kappa_r := \bigl(\cI^{\,m-r}\bK^{N}_{\bsig}\bigr)(\bff_{n+1},\bff_{n+1}),
\]

Substituting into $\cN\bu(t_{n+1})=f(t_{n+1})$ and separating the
$\bm\alpha_{n+1}$–terms gives the $d$-dimensional residual
\[
\bm{g}(\bm\alpha_{n+1})
=
\Phi_{n+1,n+1}\,\bm\alpha_{n+1}
+
\bm{c}_{0}
+
\cN_{\mathrm{nl}}\Bigl(
\bu_{\mathrm{prev}}^{(0)}+\kappa_0\bm\alpha_{n+1},\ldots,
\bu_{\mathrm{prev}}^{(m-1)}+\kappa_{m-1}\bm\alpha_{n+1}
\Bigr)
-
f(t_{n+1})
=
\bm 0,
\]
where
\[
\Phi_{n+1,n+1} := \sum_{r=0}^{m}A_r\,\kappa_r,\qquad
\bm{c}_{0} := \sum_{r=0}^{m}A_r\,\bu_{\mathrm{prev}}^{(r)}(t_{n+1})
= \sum_{r=0}^{m}A_r\,\bm{p}_{m-r-1}(t_{n+1})
+ \sum_{k=0}^{n}\Phi_{n+1,k}\,\bm\alpha_k.
\]

Setting $\cN_{\mathrm{nl}}\equiv 0$ recovers the linear closed-form
update~\eqref{eq:tt-update},
$\bm\alpha_{n+1}=\Phi_{n+1,n+1}^{-1}\bigl(f(t_{n+1})-\bm{c}_{0}\bigr)$.
We solve $\bm{g}(\bm\alpha_{n+1})=\bm 0$ by Newton's
method in $\R^{d}$,
\[
\bm\alpha_{n+1}^{(s+1)}
=
\bm\alpha_{n+1}^{(s)}
-
\bigl[\nabla_{\bm\alpha_{n+1}}\bm{g}\bigl(\bm\alpha_{n+1}^{(s)}\bigr)\bigr]^{-1}
\bm{g}\bigl(\bm\alpha_{n+1}^{(s)}\bigr),
\]
with Jacobian
\[
\nabla_{\bm\alpha_{n+1}}\bm{g}(\bm\alpha_{n+1})
=
\Phi_{n+1,n+1}
+
\sum_{r=0}^{m-1}\kappa_r\,\partial_{r}\cN_{\mathrm{nl}}\bigl(\bu^{(0)}(t_{n+1}),\ldots,\bu^{(m-1)}(t_{n+1})\bigr),
\]
where $\partial_{r}\cN_{\mathrm{nl}}$ is the Jacobian of
$\cN_{\mathrm{nl}}$ with respect to its $r$-th argument. We warm-start from
the linear-only solution
\[
\bm\alpha_{n+1}^{(0)}
=
\Phi_{n+1,n+1}^{-1}\bigl[f(t_{n+1}) - \bm{c}_{0}\bigr].
\]
Periodic full retraining is applied as specified previously.

% {\color{red}
% Erase, not used
% \smallskip\noindent\textbf{Specialization: Duffing equation.}
% For the scalar $d=1$, $m=2$ Duffing equation of
% Definition~\ref{defn:duffing} ($u''+k_1 u'+k_0 u+\gamma u^3=f$), the
% nonlinear term is $\cN_{\mathrm{nl}}(u)=\gamma u^{3}$, so
% \eqref{eq:tt-newton} reduces to the scalar cubic
% \[
% g(\alpha_{n+1})
% \;=\; \Phi_{n+1,n+1}\,\alpha_{n+1}
% \;+\; \gamma\,\bigl(u_{\mathrm{prev}}(t_{n+1}) + \kappa_0\,\alpha_{n+1}\bigr)^{3}
% \;+\; c_{0} \;-\; f(t_{n+1}) \;=\; 0,
% \]
% where $\Phi_{n+1,n+1}=\kappa_2+k_1\kappa_1+k_0\kappa_0$,
% $c_0=k_1\,u_{\mathrm{prev}}^{(1)}(t_{n+1})+k_0\,u_{\mathrm{prev}}^{(0)}(t_{n+1})+u_{\mathrm{prev}}^{(2)}(t_{n+1})$,
% and the scalar Newton step uses
% $g'(\alpha)=\Phi_{n+1,n+1}+3\gamma\kappa_0\bigl(u_{\mathrm{prev}}(t_{n+1})+\kappa_0\alpha\bigr)^{2}$.}

\subsection{Branched model training via a neural network}
In the case where we apply the neural network lift to learn the branched extension, we must jointly train to fit both the model and extension. Consider an ODE coupled with a very rough signal $f(t)$. We first construct the time-extended path $\mathbf{f} =(t,f(t))$ and feed it through the network $\Phi_\theta$ to produce the path extension $\tilde{\mathbf{f}}_{\theta}$. This extension is then used to compute the shuffle loss
\[
\mathcal L_{\mathrm{shuffle}}(\theta)
:=
\frac{1}{N}
\sum_{t_i\in\pi[0,T]}
\sum_{a,b=1}^{m}
\left|R_i^{a,b}(\theta)\right|^2,
\]
where \[
R_i^{a,b}(\theta) := D_i^a(\theta)D_i^b(\theta) - I_i^{a,b}(\theta) - I_i^{b,a}(\theta),
\]
is the second-order residual. The terms in the residual are defined as
\[
D_i^a(\theta):=\tilde{\mathbf{f}}_\theta^a(t_i)-\tilde{\mathbf{f}}_\theta^a(t_0), \quad  \text{ and } \quad I_i^{a,b}(\theta)
:=
\int_{t_0}^{t_i}
\bigl(\tilde{\mathbf{f}}_\theta^a(s)-\tilde{\mathbf{f}}_\theta^a(t_0)\bigr)\,d\tilde{\mathbf{f}}_\theta^b(s),
\]
for \(a,b\in\{1,\ldots,m\}\), \(t_i\in\pi[0,T]\), and \(N=|\pi[0,T]|\), the number of training time points.

The extended path $\Bar{\mathbf{f}}_{\theta}$ is then formed by the concatenation of the extension onto the original path i.e., \(\Bar{\mathbf{f}}:= (\mathbf{f}, \tilde{\mathbf{f}}_{\theta})\) . Next, the Kernel Gram matrices are constructed using a kernel choice and optional normalization. The linear or nonlinear collocation system is then solved using the least squares method or nonlinear L-BFGS method in order to produce the solution coefficients $\alpha$. This allows us to recover both the solution match and the forcing match from the solution and collocation ansatz. The latter is compared with the true forcing to compute the physics-informed model loss for a general (non--)linear problem ~\eqref{defn:nonlin-ode} with operator $\cN\bu=\cL\bu+\cN_{\mathrm{nl}}(\bu,\bu',\ldots,\bu^{(m-1)})$ . The loss is defined as follows.
\[
\mathcal L_{\mathrm{model}}(\bm\alpha, \theta)
\;:=\; \frac{1}{N+1}\,\bigl\lVert \;L\,\bm\alpha\;+\;\bm{\Phi}_{\mathrm{nl}}\bigl(K^{(0)}\bm\alpha,K^{(1)}\bm\alpha,\ldots,K^{(m-1)}\bm\alpha;\,\bm{p}_{m-1},\ldots,\bm{p}_{0}\bigr)\;-\;\tilde{F}_{\theta} \bigr\rVert_{2}^{2},
\]
where \(\tilde{F}_{\theta}\) is the reduced forcing, $L$ is the linear block-kernel matrix of \eqref{eq:block-kernel}, $\bm{\Phi}_{\mathrm{nl}}$ encodes the node-wise evaluation of $\cN_{\mathrm{nl}}$ at the integrated-kernel reconstructions of $\bu,\bu',\ldots,\bu^{(m-1)}$, and the Taylor polynomials $\bm{p}_{k-1}$ account for the initial data. The dependence of the loss on the parameter \(\theta\) comes from the dependence of extended path on the parameter \(\theta\).

We then define two hyperparameters \(\lambda_{\text{shuffle}}\)  and \(\lambda_{\text{model}}\) to place weights on the shuffle loss and model loss respectively. The final model loss is then defined as
\[
\cL_{\text{total}}(\bm\alpha,\theta) := \lambda_{\text{shuffle}} \mathcal L_{\mathrm{shuffle}}(\theta) + \lambda_{\text{model}}\mathcal L_{\mathrm{model}}(\bm\alpha, \theta).
\]

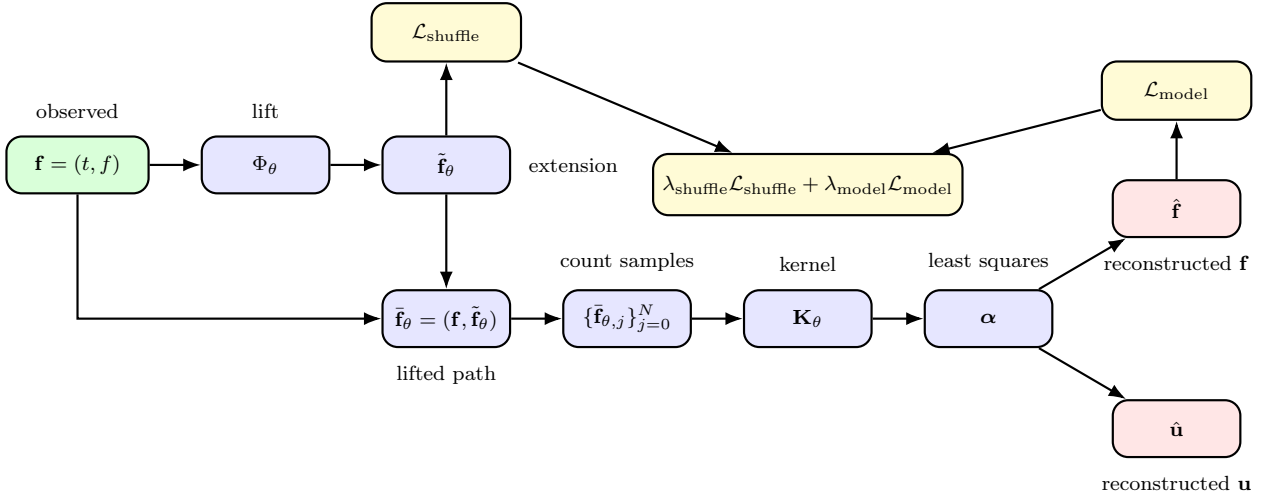
\begin{figure}[h]
\centering
\resizebox{\columnwidth}{!}{%
\begin{tikzpicture}[
    >=Latex,
    line width=0.85pt,
    font=\scriptsize,
    node distance=6mm and 7mm
]

\tikzset{
  box/.style={
    draw,
    rounded corners=6pt,
    align=center,
    minimum width=1.75cm,
    minimum height=0.78cm,
    fill=blue!10
  },
  output/.style={
    draw,
    rounded corners=6pt,
    align=center,
    minimum width=1.75cm,
    minimum height=0.78cm,
    fill=red!10
  },
  inout/.style={
    draw,
    rounded corners=6pt,
    align=center,
    minimum width=1.95cm,
    minimum height=0.78cm,
    fill=green!15
  },
  loss/.style={
    draw,
    rounded corners=6pt,
    align=center,
    minimum width=2.05cm,
    minimum height=0.80cm,
    fill=yellow!20
  },
  total/.style={
    draw,
    rounded corners=6pt,
    align=center,
    minimum width=2.15cm,
    minimum height=0.84cm,
    fill=yellow!20
  },
  lab/.style={font=\scriptsize}
}

% Main chain of nodes
\node[inout] (obs) {$\mathbf{f}=(t,f)$};
\node[box, right=of obs] (lift) {$\Phi_\theta$};
\node[box, right=of lift] (ext) {$\tilde{\mathbf{f}}_\theta$};
\node[box, below=13mm of ext] (lpath) {$\bar{\mathbf{f}}_\theta=(\mathbf{f},\tilde{\mathbf{f}}_\theta)$};
\node[box, right=of lpath] (count) {$\{\bar{\mathbf{f}}_{\theta,j}\}_{j=0}^N$};
\node[box, right=of count] (gram) {$\bK_\theta$};
\node[box, right=of gram] (alpha) {$\bm\alpha$};
\node[output, above right=7mm and 8mm of alpha] (force) {$\hat{\mathbf{f}}$};
\node[output, below right=7mm and 8mm of alpha] (urec) {$\hat{\mathbf{u}}$};

% Loss nodes
\node[loss, above=10mm of ext] (shuffle) {$\cL_{\mathrm{shuffle}}$};
\node[loss, above=8mm of force] (model) {$\cL_{\mathrm{model}}$};
\node[total, above=10mm of gram] (total) {$\lambda_{\mathrm{shuffle}}\cL_{\mathrm{shuffle}}+\lambda_{\mathrm{model}}\cL_{\mathrm{model}}$};

% Small text labels
\node[lab, above=1mm of obs] {observed};
\node[lab, above=1mm of lift] {lift};
\node[lab, right=1mm of ext] {extension};
\node[lab, below=1mm of lpath] {lifted path};
\node[lab, above=1mm of count] {count samples};
\node[lab, above=1mm of gram] {kernel};
\node[lab, above=1mm of alpha] {least squares};
\node[lab, below=1mm of force] {reconstructed $\mathbf{f}$};
\node[lab, below=1mm of urec] {reconstructed $\mathbf{u}$};

% Arrows
\draw[->] (obs) -- (lift);
\draw[->] (lift) -- (ext);
\draw[->] (ext) -- (shuffle);
\draw[->] (obs.south) |- (lpath.west);
\draw[->] (ext) -- (lpath);
\draw[->] (lpath) -- (count);
\draw[->] (count) -- (gram);
\draw[->] (gram) -- (alpha);
\draw[->] (alpha) -- (force);
\draw[->] (force) -- (model);
\draw[->] (alpha) -- (urec);
\draw[->] (shuffle) -- (total);
\draw[->] (model) -- (total);

\end{tikzpicture}%
}
\caption{Neural network lift for branched training architecture of linear ODEs}
\label{fig:linear_neural_lift_training}
\end{figure}

The parameters \(\lambda_{\text{shuffle}}\)  and \(\lambda_{\text{model}}\) are tuned so that the model learns an extension which fits the data while still following the branched (non-geometric) signature properties. Often, a higher weight is placed on the model loss than the shuffle loss so that a meaningful extension that actually aids in learning the right representative of the solution of the ODE may be selected. Illustrative diagram of the linear case is presented in Figure~\ref{fig:linear_neural_lift_training}

\subsubsection{Nonlinear Optimizations}\label{subsec:nonlinear-optimizations}
In the case of the nonlinear problem, $\alpha$ coefficients from the previous iteration are used as a warm start for the next iteration to improve convergence speed. The nonlinear diagram is shown in Figure~\ref{fig:nonlinear_neural_lift_training}. The computational bottleneck in these models is the \(\alpha\)-optimization step, so we introduce three additional optional hyperparameters to throttle and ramp this solve. The \texttt{solve\_every} parameter freezes the LBFGS updates between calls, reusing the previous \(\alpha\) for a prescribed number of ADAM steps, while the \texttt{min\_iterations} and \texttt{ramp\_portion} parameters set a low initial iteration budget and then increase it linearly up to the maximum at a chosen fraction of training. Tuning these parameters helps vastly speed up the initial iterations so that the neural network can learn expressivity, saving most lfbgs iterations for once the network is expressive. By this point, the warm starts also speed up the final iterations. For highly nonlinear problems, gaps in the frequency of LBFGS solves can cause rapidly spiking model losses. In these cases, one may opt to simply retrain for $\alpha$ every iteration and solely employ the ramping-up strategies.

\subsubsection{A note on Training Landscape}
Our joint optimization problem trains a neural network to both learning an extension while fitting the constraints of the differential equation. This is in essence the workflow followed by Physics Informed Neural Networks (PINNs). PINNs are well documented for difficulties within the training process. For example, gradient flows of the different loss components are often imbalanced, leading to optimization imbalances \cite{PINNproblems1, PINNproblems2}. This can be improved through learning rate annealing, network architectures, or gradient normalizations. In our case, specific attention is given to the weighting of shuffle and model loss so that meaningful, nontrivial extensions are learned while sufficient importance is placed on actual model fit. Learning rate schedulers are used when necessary to help re-weight the losses and ensure one does not take over the other. Stiffness and nonlinearity of problems is another essential factor in convergence difficulties \cite{PINNproblem3,PINNproblem4}. In our cases, we tackle both stiffness and nonlinearity through the use of the RBF signature kernel and tuning of the parameter $\sigma$. Tuning this parameter gives better fits for highly stiff, oscillatory problems and greatly increases convergence of the L-BFGS optimizer by improving conditioning. However, overtuning this parameter comes at the risk of overfitting the training, leading to poor performance in out of sample predictions. As of now, both the RBF hyperparameters and loss weightings are tuned to address the issues. In future work, one may consider methods introduced in the papers above to stabilize convergence for a wider range of problems.

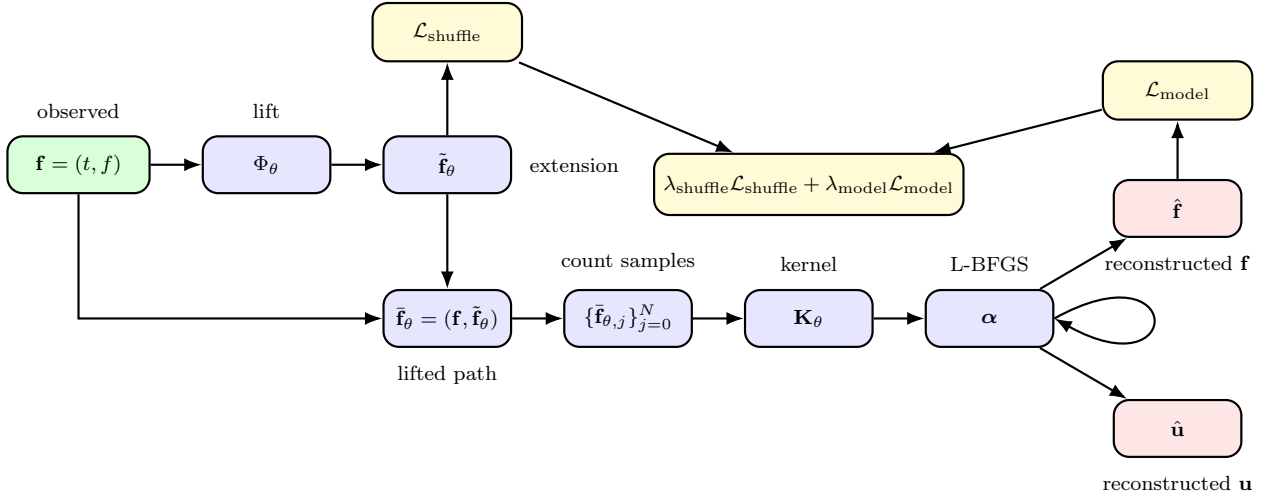
\begin{figure}[h!]
\centering
\resizebox{\columnwidth}{!}{%
\begin{tikzpicture}[
    >=Latex,
    line width=0.85pt,
    font=\scriptsize,
    node distance=6mm and 7mm
]

\tikzset{
  box/.style={
    draw,
    rounded corners=6pt,
    align=center,
    minimum width=1.75cm,
    minimum height=0.78cm,
    fill=blue!10
  },
  output/.style={
    draw,
    rounded corners=6pt,
    align=center,
    minimum width=1.75cm,
    minimum height=0.78cm,
    fill=red!10
  },
  inout/.style={
    draw,
    rounded corners=6pt,
    align=center,
    minimum width=1.95cm,
    minimum height=0.78cm,
    fill=green!15
  },
  loss/.style={
    draw,
    rounded corners=6pt,
    align=center,
    minimum width=2.05cm,
    minimum height=0.80cm,
    fill=yellow!20
  },
  total/.style={
    draw,
    rounded corners=6pt,
    align=center,
    minimum width=2.15cm,
    minimum height=0.84cm,
    fill=yellow!20
  },
  lab/.style={font=\scriptsize}
}

% Main chain of nodes
\node[inout] (obs) {$\mathbf{f}=(t,f)$};
\node[box, right=of obs] (lift) {$\Phi_\theta$};
\node[box, right=of lift] (ext) {$\tilde{\mathbf{f}}_\theta$};
\node[box, below=13mm of ext] (lpath) {$\bar{\mathbf{f}}_\theta=(\mathbf{f},\tilde{\mathbf{f}}_\theta)$};
\node[box, right=of lpath] (count) {$\{\bar{\mathbf{f}}_{\theta,j}\}_{j=0}^N$};
\node[box, right=of count] (gram) {$\bK_\theta$};
\node[box, right=of gram] (alpha) {$\bm\alpha$};
\node[output, above right=7mm and 8mm of alpha] (force) {$\hat{\mathbf{f}}$};
\node[output, below right=7mm and 8mm of alpha] (urec) {$\hat{\mathbf{u}}$};

% Loss nodes
\node[loss, above=10mm of ext] (shuffle) {$\cL_{\mathrm{shuffle}}$};
\node[loss, above=8mm of force] (model) {$\cL_{\mathrm{model}}$};
\node[total, above=10mm of gram] (total) {$\lambda_{\mathrm{shuffle}}\cL_{\mathrm{shuffle}}+\lambda_{\mathrm{model}}\cL_{\mathrm{model}}$};

% Small text labels
\node[lab, above=1mm of obs] {observed};
\node[lab, above=1mm of lift] {lift};
\node[lab, right=1mm of ext] {extension};
\node[lab, below=1mm of lpath] {lifted path};
\node[lab, above=1mm of count] {count samples};
\node[lab, above=1mm of gram] {kernel};
\node[lab, above=1mm of alpha] {L-BFGS};
\node[lab, below=1mm of force] {reconstructed \(\mathbf{f}\)};
\node[lab, below=1mm of urec] {reconstructed \(\mathbf{u}\)};

% Arrows
\draw[->] (obs) -- (lift);
\draw[->] (lift) -- (ext);
\draw[->] (ext) -- (shuffle);
\draw[->] (obs.south) |- (lpath.west);
\draw[->] (ext) -- (lpath);
\draw[->] (lpath) -- (count);
\draw[->] (count) -- (gram);
\draw[->] (gram) -- (alpha);
\draw[->] (alpha) -- (force);
\draw[->] (force) -- (model);
\draw[->] (alpha) -- (urec);
\draw[->] (shuffle) -- (total);
\draw[->] (model) -- (total);

% Self-loop on alpha
\draw[->] (alpha.east) .. controls +(18mm,10mm) and +(18mm,-10mm) .. (alpha.east);

\end{tikzpicture}%
}
\caption{Neural network lift for branched training architecture of nonlinear ODEs}
\label{fig:nonlinear_neural_lift_training}
\end{figure}

%%%%%%%%%%%%%%%%%%%%%%%%%%%%%%%%%%

\section{Numerical Experiments}\label{sec:numerical-experiments}
We test the branched signature kernel solver developed in
Section~\ref{sec: kernel solvers} and Section~\ref{sec:test-train} on a range of examples. In particular, we focus on the setting where only a single trajectory of the forcing term is observed, motivated by real-world scenarios in which the forcing signal is
typically noisy. We compare the performance of the geometric signature kernel solver and the branched signature kernel solver when the forcing term
is highly rough, such as a sample path of fractional Brownian motion. Branched extensions are constructed in the form of the \(t\)-value lift and neural network extension introduced in Definition~\ref{def: tlift} and Definition~\ref{def: NN Lift}. These extension methods are compared within experiments, highlighting problems where the T-lift is adequate and demonstrating that in more noisy or complex cases, the neural network extension is needed.  Furthermore, both methods from Section~\ref{subsec:lin-inte} and Section~\ref{subsec:lin-inte-2} are compared.

Across the experiments, we test and consider variations of the truncation level $N$ of the signature, kernel matrix normalization (none or robust) and kernel choices (linear or RBF). We note that within each example, hyper parameters such as regression regularization, L-BFGS tolerance/iteration count, and the RBF sigma value must be tuned correctly. For each section, time lifts are defined by adding a path component of $t^\alpha$, where $\alpha = 2H$ for the holder value H of the underlying generated forcing. The branched neural networks parameters are tuned for each problem for optimal convergence. Each model uses varying hidden depths/layers for an MLP with a tanh activation function, using Xavier initialization for weights. Integrations of the Kernel Gram Matrices are done using the cumulative trapezoidal integration method unless otherwise stated. Ground truth solutions are generated by either high order Runge-Kutta integration or implicit multi-step solvers implemented within the \texttt{scipy} library.

\subsection{Earthquake displacement model: El-Centro calibration}\label{subsec:eq-model}

We apply the method from Section~\ref{subsec:lin-inte-2} to the 1940 El-Centro strong-motion record sourced from \cite{elcentroearthquakedata}. Structural response is modeled by the damped oscillator, a special case of Example~\ref{ex:scalar-integral}:
\[
\ddot u(t) + 2\xi\omega_n\,\dot u(t) + \omega_n^{2}u(t)
        = -9.81\,a(t),
\qquad
u(0)=\dot u(0)=0,
\]
where \(a(t)\) is the normalized ground-acceleration record, \(\xi\) is the damping ratio, and \(\omega_n\) is the natural frequency. In the experiment, we take \(\xi=0.02\), \(T_n=5\,\mathrm{s}\), and \(\omega_n=2\pi/T_n\). First calibration is done using a signature depth of 12, a linear kernel, and the robust normalization procedure. Figure~\ref{fig:el_centro_example}, displays fits for both the double integrated forcing and solution approximations. The relative MSE for the double integrated forcing was 2.28e-4 while the relative MSE for the solution was 1.47e-04, demonstrating the baseline algorithms effectiveness for moderately rough signals when a high signature depth is used.

\begin{figure}[H]
  \centering
  \captionsetup{font=small,skip=4pt}
\includegraphics[width=1\linewidth]{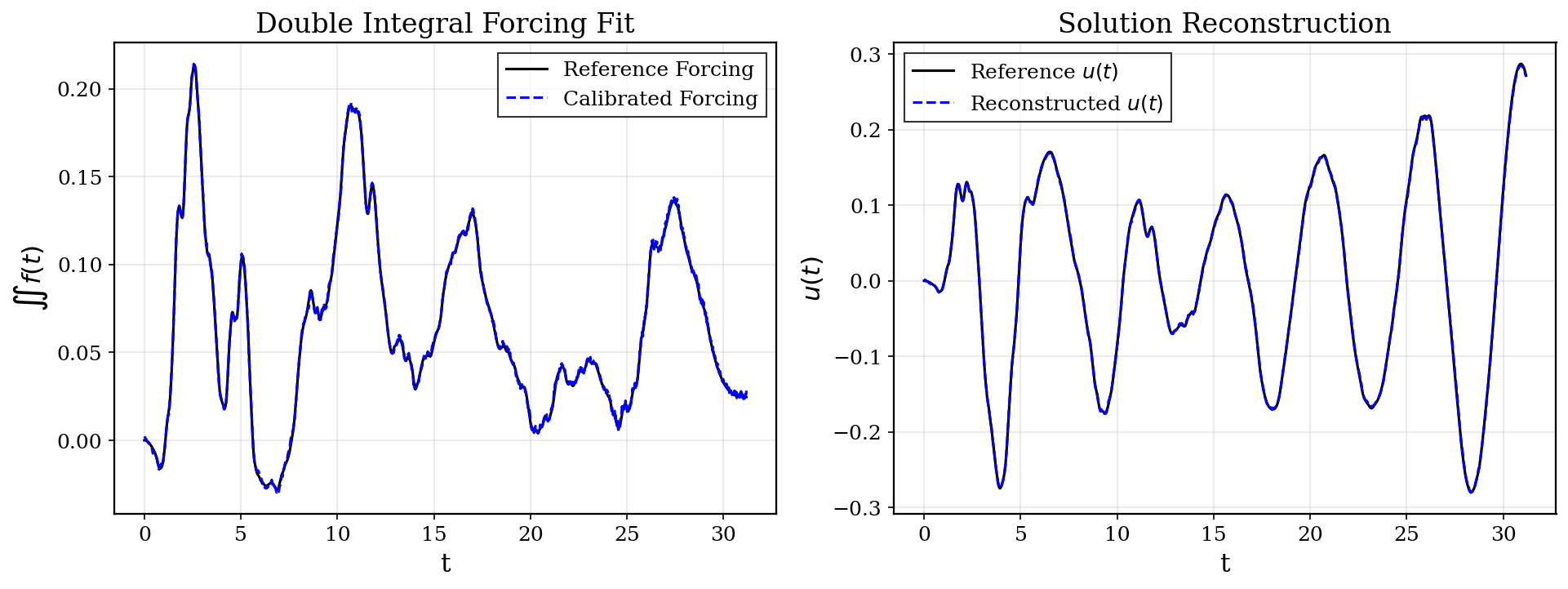}
  \caption{El-Centro earthquake calibration. Left: forcing calibration-learned kernel representation against $-9.81\,a(t)$. Right: displacement-predicted $\hat{u}(t)$ vs.\ reference $u(t)$.}
  \label{fig:el_centro_example}
\end{figure}
\vspace{-0.75em}

Next, training and testing were applied on a 40/60 split using the rolling-retrain protocol of Section~\ref{sec:test-train}, with periodic retraining after 5 predictions. The baseline model is compared to a t-lift model with $h=.25$ The signature depth is lowered to 5 in order to show the effect of the branched model. The results are summarized in Table~\ref{tab:6.1results}.

\begin{table}[ht]
\centering
\begin{tabular}{lcccc}
\toprule
& \multicolumn{2}{c}{\textbf{Training}} & \multicolumn{2}{c}{\textbf{Testing}} \\
\cmidrule(lr){2-3} \cmidrule(lr){4-5}
& Integrated Forcing & Solution & Integrated Forcing & Solution \\
\midrule
Default Kernel        & 4.29e-03 & 4.05e-03 & 1.84e-16 & 1.21e+00 \\
\(t\)-lift branched   & \textbf{5.68e-04} & \textbf{4.63e-04} & \textbf{2.53e-21} & \textbf{1.19e-01} \\
\bottomrule
\end{tabular}
\caption{Match results for the Default Kernel model and the \(t\)-lift branched model.}
\label{tab:6.1results}
\end{table}

The baseline model fails to accurately predict given the rough signals, while the t-lift greatly improves performance. The t-lift still fails to handle all of the roughness, however, suggesting another extension, such as the neural network may be needed. We note the reason for the testing forcing being fit better than the training forcing is that the training forcing calibrates directly to all of the training data at once, whereas in the rolling retrain, we get new $\alpha$ values one by one and match them exactly to the new forcing values. We refer to Figure~\ref{fig:el_centro_prediction} for plots of both models.

\begin{figure}[H]
  \centering
  \captionsetup{font=small,skip=4pt}
  \includegraphics[width=1\linewidth]{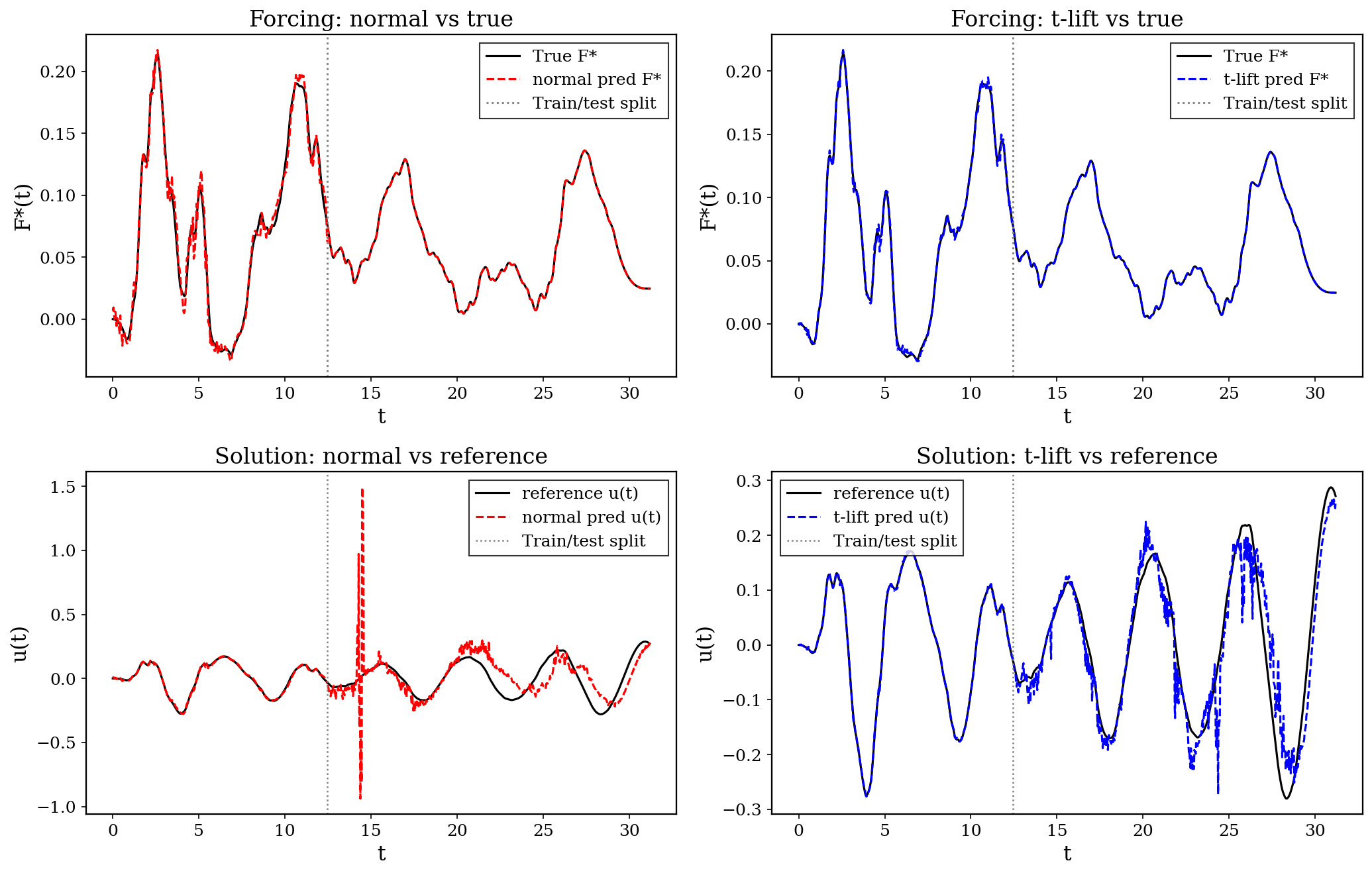}
  \caption{El-Centro streaming prediction. Top row: forcing calibration; bottom row: solution reconstruction. Left: non-branched path $(t,\,-9.81\,a(t))$. Right: branched path $(t,\,t^{\alpha},\,-9.81\,a(t))$.}
  \label{fig:el_centro_prediction}
\end{figure}
%\vspace{-0.75em}

\subsection{Solow capital-stock model}\label{subsec:solow} 

For this example, we use a modified version of the Solow capital-stock growth model \cite{SolowModel}, where we use real GDP from FRED \cite{fred_gdp} as the forcing term.
The Solow capital-stock dynamics are governed by the first-order linear ODE
\[
\dot Y(t) \;+\; \delta\,Y(t) \;=\; s\,F(t),\qquad Y(0)=Y_0,
\]
with depreciation rate $\delta\in[0,1]$, savings rate $s\in[0,1]$, GDP forcing $F(t)$, and initial capital stock $Y_0>0$. In the notation of Definition~\ref{defn:lin-ode}, this is the $d=1$, $m=1$ instance with $A_1=1$, $A_0=\delta$, forcing $sF(t)$, and initial data $g_0=Y_0$. In our experiment, we set an initial condition $Y_0=3.1$. We apply the linear integral method of
Section~\ref{subsec:lin-inte-2} with no normalization and signature
truncation level $N=3$. First consider the results of calibration on the entire interval, shown in Figure~\ref{fig:solow_example}. The relative integrated forcing and solution errors were 6.59e-08 and 1.14e-06  respectively. The lower truncation level here reflects the smoother, first-order nature of the Solow ODE in comparsion to the el-centro problem, allowing for lower level signatures to capture the dynamics.

\begin{figure}[H]
  \centering
  \captionsetup{font=small,skip=4pt}
  \includegraphics[width=0.9\linewidth]{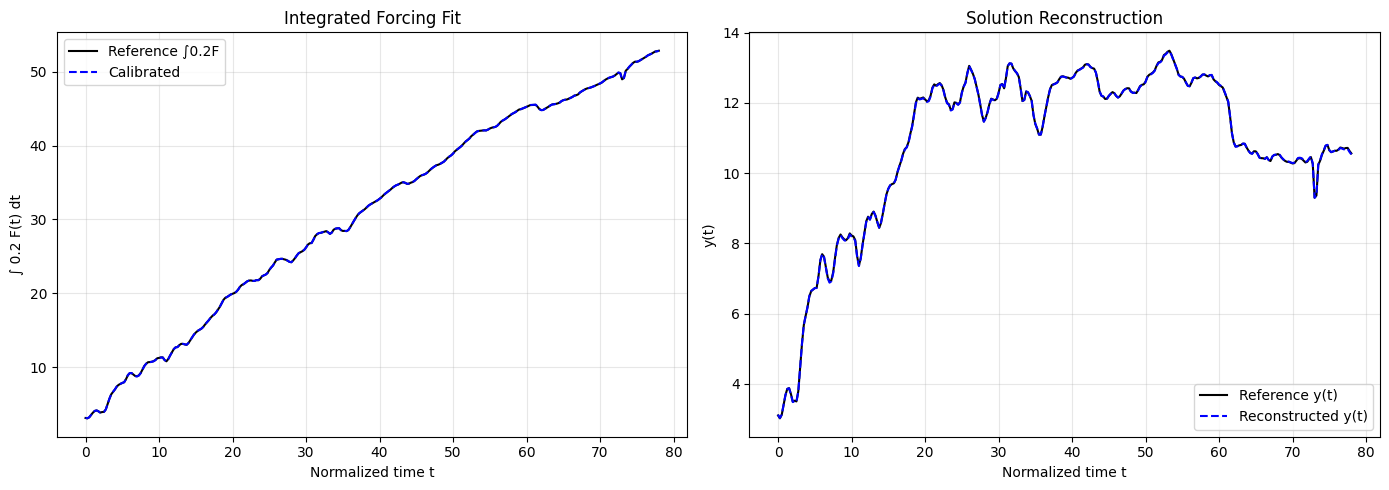}
  \caption{Solow ODE calibration. Left: GDP forcing match. Right: capital-stock solution match.}
  \label{fig:solow_example}
\end{figure}
\vspace{-0.75em}

Next, a testing and training experiment is done using the rolling retrain protocol, using the same hyperparameters as before. The first 50 data points are used for training, with periodic retraining applied every 10 steps. The results in Figure~\ref{fig:solow_retrain_example} and Table~\ref{tab:6.2results} confirm the low depth signatures capabilities in this case. 

\begin{figure}[H]
  \centering
  \captionsetup{font=small,skip=4pt}
  \includegraphics[width=0.9\linewidth]{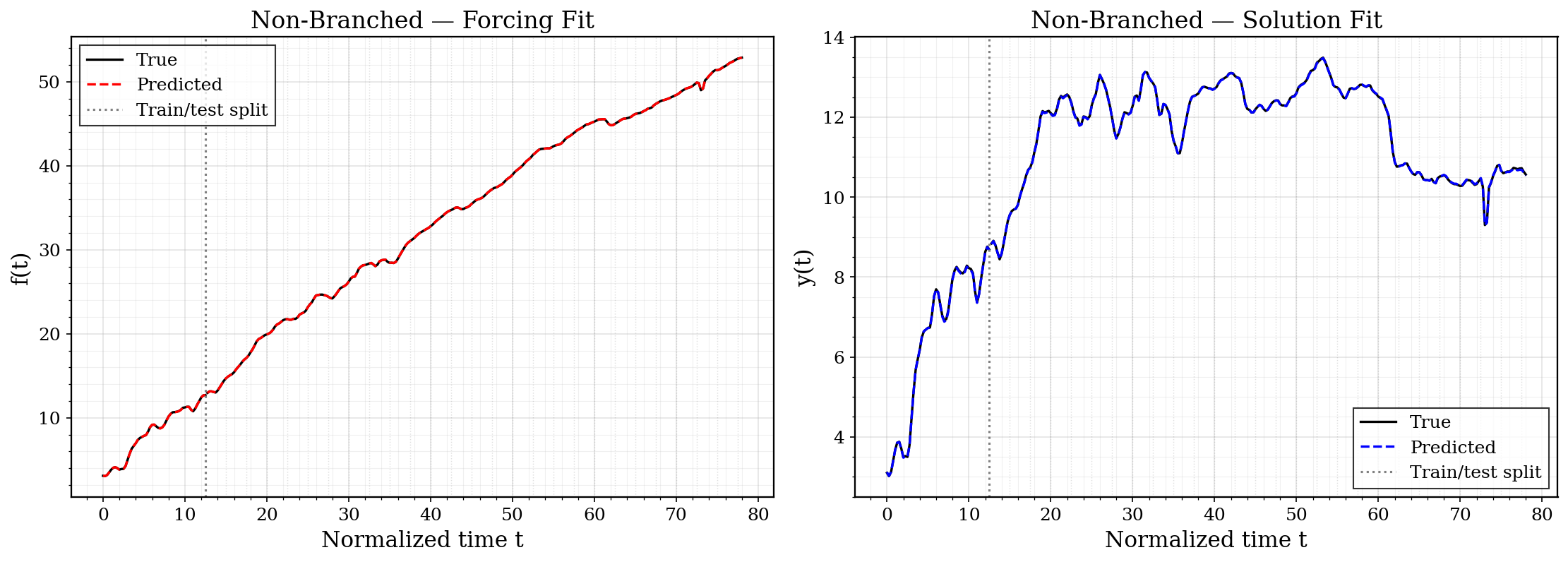}
  \caption{Solow streaming prediction with the rolling retrain protocol.}
  \label{fig:solow_retrain_example}
\end{figure}
\vspace{-0.75em}

\begin{table}[H]
\centering
\begin{tabular}{lcc}
\hline
 & Training & Testing \\
\hline
Forcing  & 2.91e-9 & 6.53e-29 \\
Solution & 1.91e-06 & 1.04e-06 \\
\hline
\end{tabular}
\caption{Results for the rolling-retrain Solow experiment}
\label{tab:6.2results}
\end{table}

\subsection{Linear ODE with fBM Forcing Example} 
Consider a linear second order ODE (SDOF equation) driven by a fractional Brownian motion forcing term. We use the following governing equation.
\[
m\ddot{U} + c\dot{U} + k U = F(t), \qquad U(0) =a, \quad U'(0) =b,
\] 
In our simulation, $F(t)$ is a 3000 point fBM on the interval from 0 to 1 with hurst parameter $H=.25$. \(m,c,\) and \(k\) are set at 1, 10, and 5 respectively. We compare training and prediction performance across the baseline signature model, t-lift, and neural lift models, also comparing across methods from Section~\ref{subsec:lin-inte} and Section~\ref{subsec:lin-inte-2}, refered to as method 1 and method 2, respectively. For each method, an RBF signature kernel with signature depth 3 and $\sigma=1$ is employed along with the robust normalization method. The standard training and prediction method are applied to a 90/10 test train split, with periodic retraining every 2 steps. For our neural lift, 3 extensions are trained with hidden depths $(2, 4, 8, 4, 2)$ over 1500 ADAM epochs. Overall learning rate, model loss weight, and shuffle loss weight were set to be 5e-3, 10, and 1e-4 respectively, with a learning rate scheduler reducing losses for every 300 stale iterations.

We first refer to Table~\ref{tab:6.3table1} for comparisons of the methods. The first method vastly outperforms the second, likely due to the double integration of the forcing introducing approximation errors. The second method also fails to predict well, showing inconsistent results. For problems with such rough forcings or higher orders of integrations required, it is therefore recommended to use Method 1.

\begin{table}[ht]
\centering
\begin{tabular}{lcccccc}
\toprule
& \multicolumn{3}{c}{\textbf{Method 1}} & \multicolumn{3}{c}{\textbf{Method 2}} \\
\cmidrule(lr){2-4} \cmidrule(lr){5-7}
& Baseline & \(t\)-lift & Branched & Baseline & \(t\)-lift & Branched \\
\midrule
Training Solution & 4.56e-11 & 2.15e-11 & \textbf{1.22e-11} & 1.12e-06 & 2.68e-07 & \textbf{5.46e-09} \\
Testing Solution  & 4.76e-09 & 3.18e-11 & \textbf{1.17e-13} & 6.24e-03 & \textbf{6.01e-03} & 3.55e-02 \\
\bottomrule
\end{tabular}
\caption{Relative MSE of the solution $u$ across all variants and methods on the training and testing splits.}
\label{tab:6.3table1}
\end{table}

Next, we compare the results of Method 1 on the calibration and prediction task. Results are shown in Table~\ref{tab:6.3table2}. Across the training and testing sections, both branched models outperform the baseline model. The neural network extension performs magnitudes better than the t lift in calibration, and performs better in predicting the solution as well. This suggests the neural branched model is able to sufficiently capture the branched structure within the rough paths. We refer to Figure~\ref{fig:6.3fig1} for results of the neural model on testing and prediction.

\begin{table}[ht]
\centering
\begin{tabular}{lcccc}
\toprule
& \multicolumn{2}{c}{\textbf{Training}} & \multicolumn{2}{c}{\textbf{Testing}} \\
\cmidrule(lr){2-3} \cmidrule(lr){4-5}
& Forcing & Solution & Forcing & Solution \\
\midrule
Baseline   & 1.89e-04 & 4.56e-11 & 1.96e-01 & 4.76e-09 \\
t-lift Solution   & 2.57e-05 & 2.15e-11 & \textbf{3.32e-03} & 3.18e-11 \\
Branched Solution & \textbf{3.36e-07} & \textbf{1.22e-11} & 8.91e-03 & \textbf{1.17e-13} \\
\bottomrule
\end{tabular}
\caption{Relative MSE of the forcing $f$ and solution $u$ for Method 1 (direct forcing match) across all variants on the training and testing splits.}
\label{tab:6.3table2}
\end{table}

\begin{figure}[h]
    \centering
    \includegraphics[width=1\linewidth]{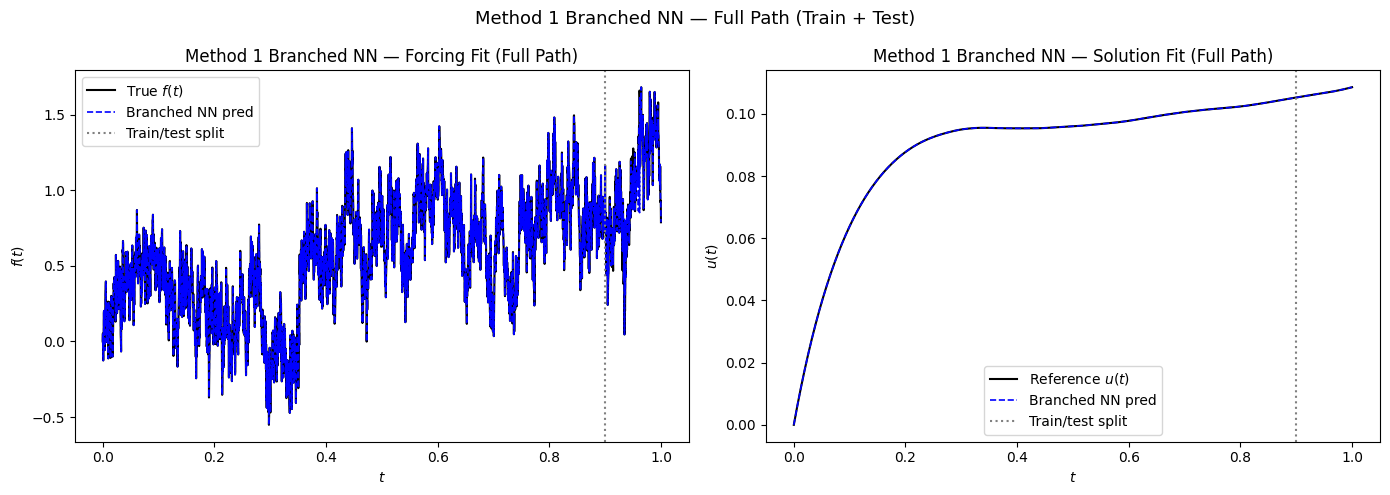}
    \caption{Forcing and Solution calibration and Prediction for a linear ODE driven by FBM.}
    \label{fig:6.3fig1}
\end{figure}

\vspace{-0.75em}

\subsection{Modeling Nonlinear ODE driven by FBM (Duffing Model)}\label{subsec:duffing}

The Duffing oscillator extends the second-order SDOF system with a nonlinear cubic term and has broad applications across many fields. It is used to model nonlinear vibrations in timber structures \cite{timberduffingmodel}, energy harvesting under magnetic forcing and its noisy extensions \cite{magnet_duffing,magnet_duffing_noisy}, financial asset dynamics with fractional Gaussian processes \cite{duffing_FGN_modeling}, and complex signal processing in high-noise environments \cite{signal_processing_duffing}. Duffing oscillators driven by Gaussian noise are also employed for stochastic resonance and weak-signal detection, including mechanical fault diagnosis such as rolling bearing and rotor faults \cite{stochastic_duffing_detection,SR_duffing}, and have been further developed within stochastic modeling frameworks \cite{stoch_duffing}.

We apply our branched signature kernel method to solve Duffing oscillators of the standard form with Neumann and Dirichlet boundary conditions as specified in Example \ref{defn:duffing} using the nonlinear approximation method described in \ref{subsec:nonlin-inte}. Our time values consist of 3000 points over the interval 0 to 1. We set $a=0, b=1, k_0=5, k_1=10, \gamma=10$ as our constants and coefficients. The forcing is taken as a fractional Brownian motion with hurst parameter $H=.25$. The first 90 percent of values are used for training while testing is done on the last 10 percent. A signature depth of 3 is used together with column-wise robust normalization. The kernel is chosen to be radial basis with bandwidth $\sigma=3$. The standard prediction procedure is used, with retraining performed every 2 steps. The maximum LBFGS iteration is set to 300. The branched architecture employs hidden layer widths $(2, 4, 8, 4, 2)$ and learns 2 extension channels. Optimization proceeds using ADAM for 2000 iterations with learning rate $3e-4$ consisting of model loss weight $10$ and shuffle loss weight $10e-3$. We use the acceleration procedures from \ref{subsec:nonlinear-optimizations}, solving for $\alpha$ every 100 ADAM iterations while ramping up from 50 to 300 LFBGS iterations over the first 25 percent of ADAM iterations.

The results for both the baseline and branched model are summarized in Table~\ref{tab:6.4table1}. While both models fit extremely well in and out of sample, the branched model showed significant improvement over the baseline signature kernel model. This demonstrates the neural networks ability to encode the branched information not captured by the classical signature. The results for the branched models testing and training fits are presented in Figure~\ref{fig:6.4figure1}

\begin{table}[ht]
\centering
\begin{tabular}{lcccc}
\toprule
& \multicolumn{2}{c}{\textbf{Training}} & \multicolumn{2}{c}{\textbf{Testing}} \\
\cmidrule(lr){2-3} \cmidrule(lr){4-5}
& Forcing & Solution & Forcing & Solution \\
\midrule
Non-branched & 4.60e-02 & 7.86e-05 & 4.88e-03 & 9.31e-08 \\
Branched     & \textbf{1.80e-02} & \textbf{1.66e-05} & \textbf{3.42e-03} & \textbf{4.79e-08} \\
\bottomrule
\end{tabular}
\caption{Relative MSE of the forcing target and solution $u$ for the Duffing model}
\label{tab:6.4table1}
\end{table}

\begin{figure}
    \centering
    \includegraphics[width=1\linewidth]{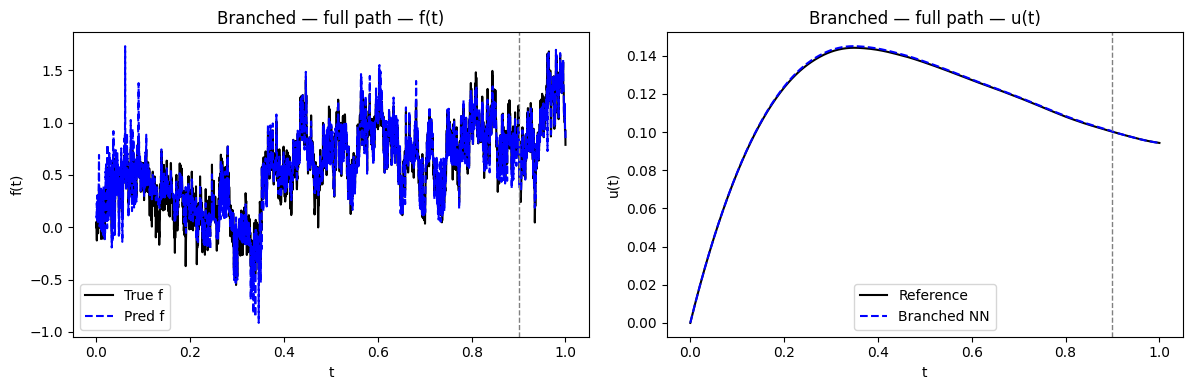}
    \caption{Branched Signature Kernel Model fit for the Duffing Equation}
    \label{fig:6.4figure1}
\end{figure}

\subsection{Variable Coefficient ODE: Arias-Intensity degraded SDOF}

We have considered the case of linear ODEs with constant coefficients. However, we may also examine equations in which the coefficients depend on the time or forcing themselves. 
\begin{defn}[Variable Coefficient Second Order Linear ODE]\label{defn:vc_ode}
\begin{align*}\color{red}
&k_1(t,f(t)) u''(t) + k_2(t,f(t)) u'(t) + k_3(t,f(t)) u(t)=f(t), \quad u(0)=a,\quad
u'(0)=b,
\end{align*}
with functions $k_1,k_2,k_3$ and $x\in[0,T]$.
\end{defn}

Forms of the SDOF system with time-varying coefficients have been studied in \cite{variablecoefficient1,variablecoefficient2,variablecoefficient3}. As long as the coefficients do not depend on 
$u(t)$ or its derivatives, the ODE reamins linear and the signature-kernel algorithm is unchanged: one simply evaluates the coefficient functions on the grid when forming the alphas and kernel matrix ansatz. Our solver therefore applies directly to ODEs with path-dependent coefficients. We illustrate this with an example from earthquake engineering. The Arias intensity \cite{Arias1970measure}
\[
I_A(t)\;:=\;\frac{\pi}{2g}\int_{0}^{t}a(s)^{2}\,\d s
\]
is a path functional of the ground acceleration that quantifies the cumulative energy delivered to the structure. Given a damage source $D(t)$, effective stiffness is often modeled as decaying with damage \cite{ramtani2013contmechanics}:
$$
k_{eff}(t) = (1-D(t))k_0,
\qquad
\frac{k_{eff}(t)}{m} = \frac{(1-D(t))k_0}{m} = (1-D(t))\omega_0^2.
$$

As the Arias intensity strongly correlates with destructive potential \cite{cabanas1997Ariasdestruction}, we introduce a degradation rate \(\delta\) and use \(I_A(t)\) as the damage variable, yielding the variable stiffness coefficient
\[
A_0(t)=\omega_n^{\,2}\bigl(1-\delta\,I_A(t)\bigr)_{+},
\]
where \((\cdot)_{+}:=\max(\cdot,0)\) enforces non-negative stiffness. Substituting this into Definition~\ref{defn:vc_ode} gives the Arias-intensity degraded SDOF model:
\begin{defn}[Arias-intensity degraded SDOF]\label{defn:Arias-sdof}
\[
\ddot u(t) \;+\; 2\xi\omega_n\,\dot u(t)
\;+\; A_0(t)\,u(t)
\;=\; -9.81\,a(t),
\qquad u(0)=0,\quad \dot u(0)=0,
\]
with degradation rate \(\delta\ge 0\). The coefficient \(A_0(t)\) depends on the forcing path through the cumulative integral \(I_A(t)\).
\end{defn}

In the collocation framework of \eqref{eq:block-kernel}, the block matrix becomes
\[
L_{ji} \;=\; A_2\,K_{ji} \;+\; A_1\,K^{(1)}_{ji}
\;+\; A_0(t_j)\,K^{(2)}_{ji},
\]
where \(I_A(t_j)=\tfrac{\pi}{2g}\sum_{k=0}^{j-1}a(t_k)^{2}(t_{k+1}-t_k)\) is computed once from the count-sampled forcing path \(\bff_j=(a(t_0),\ldots,a(t_j))\). The integral method and calibration protocol of Section~\ref{subsec:lin-inte} and Section~\ref{sec:test-train} remain unchanged.

For the experiment, calibration is compared across the baseline and branched signature kernel models, along with a solver using a point-wise rbf kernel with no signatures. We consider an example in which the acceleration $a(t)$ is given by a fractional brownian motion with hurst parameter $H=.2$. We set $\omega_n = 2\pi, \xi = .05, \delta = 5$. 1000 observations are taken over the interval from 0 to 1. Refer to Figure~\ref{fig:Ariassetup} for plots of the forcing, Arias intensity, and effective stiffness coefficient.

\begin{figure}
    \centering
    \includegraphics[width=1\linewidth]{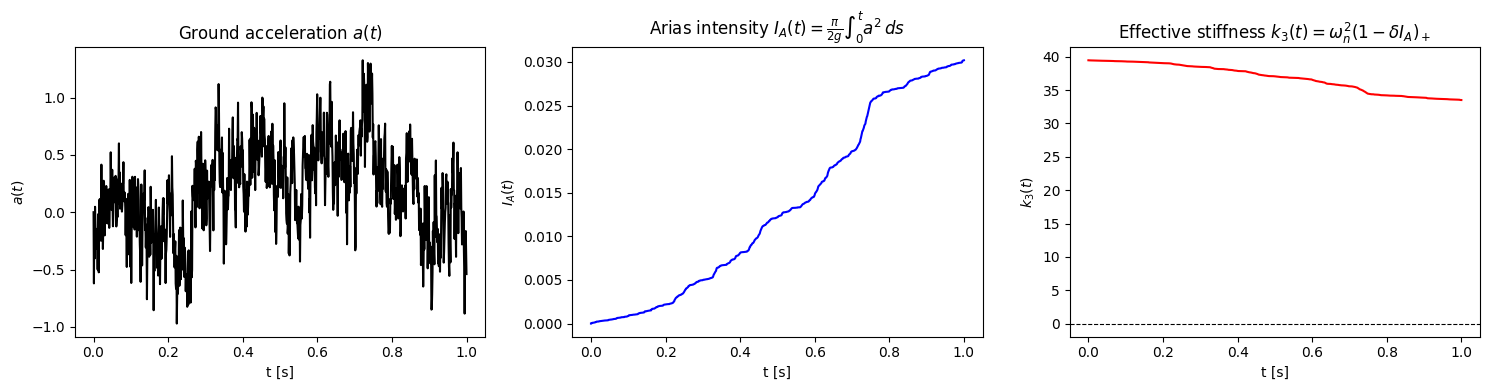}
    \caption{plots of generated Forcing, Arias Intensity, and Stiffness Coefficient}
    \label{fig:Ariassetup}
\end{figure}

For each case, a radial basis kernel with $\sigma=1$ is used. Signature solvers use a signature depth of 2 together with a column‑wise robust normalization. The branched model trains 3 extensions using hidden layers $(4, 8, 16, 8, 4, 2)$. Optimization is done for 2000 epochs, employing the ADAM algorithm with learning rate $10e-4$, a model loss weighted by 10, and a shuffle regularization term weighted by $10e-4$. A plateau-based learning rate schedule reduces the step size by a factor of 0.5 after 500 epochs without improvement. The results for the three models are summarized in Table~\ref{tab:Ariasresults} and plotted in Figure~\ref{fig:Ariascalibration}.

\begin{table}[ht]
    \centering
    \begin{tabular}{lcc}
        \hline
        Model & Forcing & Solution \\
        \hline
        RBF (No Signatures) Kernel Model           & $5.49\mathrm{e}{-1}$  & $1.17\mathrm{e}{-3}$  \\
        Non-branched Signature Kernel Model  & $2.94\mathrm{e}{-4}$  & $1.46\mathrm{e}{-8}$  \\
        Branched Signature Kernel Model      & $\mathbf{6.34\mathrm{e}{-6}}$  & $\mathbf{3.37\mathrm{e}{-11}}$ \\
        \hline
    \end{tabular}
    \caption{Relative MSE figures for the Arias Intensity Degraded SDOF Problem}
    \label{tab:Ariasresults}
\end{table}

The base RBF Kernel fails to adequately model the rough forcing term and misses some of the shaping within the solution. The base signature model sees better results, but is unable to capture all of the branched information within the forcing path. The Branched Model shows the best performance out of all models by significant orders of magnitudes. We note that for the point-wise RBF model, the $\sigma$ parameter may be "over-tuned", i.e. shrunk extremely small so that the kernel matches the forcing's roughness. However, this overfitting leads to a complete inability to predict out of sample. The branched signature kernel model remains able to fit both the forcing and predict out of sample without overfitting to the data, as seen in Example~\ref{subsec:duffing}

\begin{figure}
    \centering
    \includegraphics[width=1\linewidth]{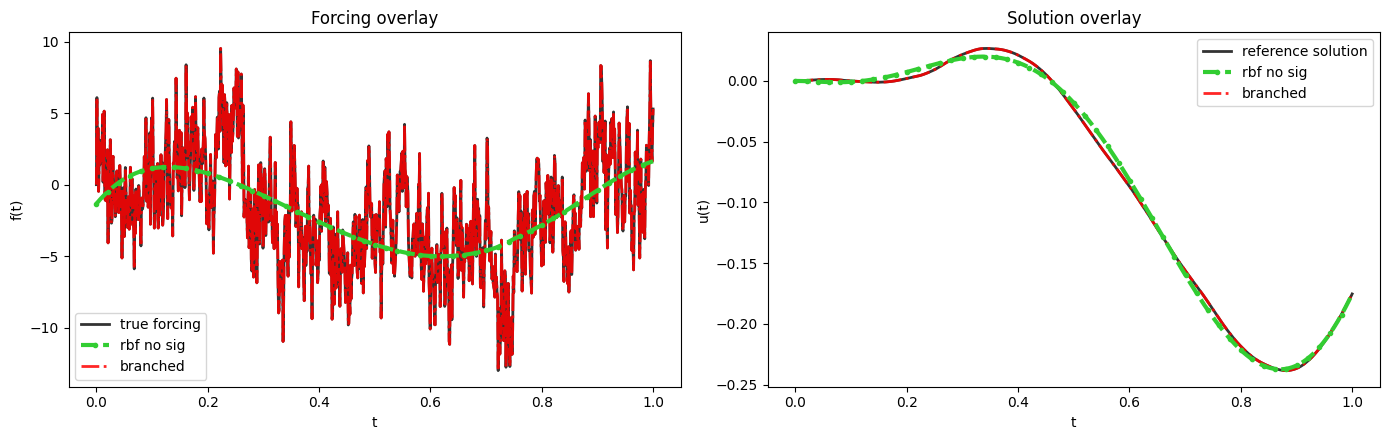}
    \caption{Calibration results for the branched signature kernel and a point-wise (no signatures) RBF kernel on the Arias Intensity degraded SDOF equation}
    \label{fig:Ariascalibration}
\end{figure}

\subsection{Noisy Kuramoto oscillator system}

The general Kuramoto oscillator system \cite{10.1007/BFb0013365} on $n$ coupled phase oscillators is given by
\begin{equation}
\frac{d\theta_i}{dt} \;=\; \omega_i + \frac{1}{n}\sum_{j=1}^{n} K_{ij}\, \sin(\theta_j - \theta_i), \qquad i = 1, \ldots, n,
\label{eq:kuramoto}
\end{equation}
where $\theta_i(t) \in \mathbb{T}$ is the phase of oscillator $i$, $\omega_i \in \mathbb{R}$ its natural frequency, and $K_{ij} \in \mathbb{R}$ the coupling strength. Here $n$ is the number of oscillators, distinct from the count-sampled path count $N$ used elsewhere. Introducing a uniform coupling constant $K$ and additive white noise $\zeta_i$ yields the noisy system \cite{sakaguchi1988cooperative}, which can be written either as an ODE or in differential form:
\[
\frac{d\theta_i}{dt}
= \omega_i + \zeta_i + \frac{K}{n}\sum_{j=1}^{n} \sin(\theta_j - \theta_i)
\qquad\text{or}\qquad
d\theta_i(t)
= \Bigl[ \omega_i + \frac{K}{n}\sum_{j=1}^{n} \sin(\theta_j - \theta_i) \Bigr] dt + d B_i^H(t),
\]
where $B_i^H(t)$ is a fractional Brownian motion with hurst parameter $H$. We may write its ODE form as:
$$\frac{d\theta_i}{dt} = \omega_i + \frac{K}{n}\sum_{j=1}^{n} \sin(\theta_j - \theta_i) + \zeta_i(t),\quad \text{or}\quad
\frac{d\theta_i}{dt} - \omega_i - \frac{K}{n}\sum_{j=1}^{n} \sin(\theta_j - \theta_i) = \zeta_i(t).$$

In the first form, one could observe the right-hand side and use the signature-kernel algorithm to calibrate it and recover $\theta$, a linear task in which the branched signature captures the fractional Gaussian noise. Here we instead analyze the second form: we observe the noise and use the nonlinear approximation algorithm to calibrate the left-hand side to the right, enforcing the physical law.

First, let us describe how to generate the required data and reference solution. We fix $n$, frequencies $\omega_i$ for $i=1,\dots,n$, and coupling constant $K$, choose a uniform grid $0 = t_0, \dots, t_{N-1} = T$, and set the initial phase vector $\theta(0)=\bigl(\theta_1(0),\ldots,\theta_n(0)\bigr)$. For each oscillator \(i\), we generate a fractional Brownian motion sample path $B_i^H(t_0), \dots, B_i^H(t_{N-1})$ with Hurst parameter \(H\). Using increments $\Delta B_i^H(t_k):=B_i^H(t_{k+1})-B_i^H(t_k)$ and time step $\Delta t=t_{k+1}-t_k$, $k = 0,\dots,N-2$, we construct the Euler approximation
\[
\theta_i(t_{k+1}) = \theta_i(t_k) + \Bigl[ \omega_i + \frac{K}{n}\sum_{j=1}^{n}\sin\bigl(\theta_j(t_k)-\theta_i(t_k)\bigr) \Bigr]\Delta t + \Delta B_i^H(t_k),
\]
for \(k = 0,\dots,N-2\) and \(i = 1,\dots,n\). For the signature-kernel solver via \ref{subsec:lin-inte}, we take the raw path and fractional Gaussian noise approximated by
\[
\mathbf{f}_k = \bigl(t_k,\eta_1^H(t_k),\eta_2^H(t_k),\ldots,\eta_n^H(t_k)\bigr), \qquad \eta_i^H(t_k) \approx \frac{\Delta B_i^H(t_k)}{\Delta t}.
\]

For our simulation, we set $n=3$ oscillators with frequencies of $-1, -.3$, and $1.5$, respectively. The coupling constant is set to $k=3$. Each oscillator is driven by a $3000$ point fractional Brownian motion with Hurst parameter $H=0.4$. Initial values of $\theta$ are selected randomly along the range $[0,2\pi)$. Signature depth for the models is set at 3, and robust normalization is used along with an RBF kernel with parameter $\sigma=1$ and LBFGS iteration count of 400. For this example, trapezoid integration is replaced by a left point Reimann sum as to be consistent with the generation scheme. The branched model trains 2 extensions with hidden layers $(8, 16, 32, 16, 8)$. Optimization is done with 1000 ADAM iterations with an overall learning rate of $3e-4$, model loss weight of $1e3$, and shuffle loss weight of $1e-3$. The model loss is set higher in this case to compensate for the contribution of all 3 generated paths to the shuffle loss. Acceleration is applied through the methods in \ref{subsec:nonlinear-optimizations}. Due to the highly nonlinear sine term causing spikes in losses, we choose to optimize for $\alpha$ at every ADAM iteration. LFBGS Iterations start off at 20 and linearly ramp up to the maximum of 400 at 75 percent of the iterations.

Results showing fits of the models for the forcing, derivative of the solution, and solution for the branched and non branched models are presented in Table~\ref{tab:relative_mse_kuramoto}. Plots of the branched model fits for each oscillators forcing and solution are presented in Figure~\ref{fig:kuramoto}. Both models fit all oscillators extremely well. However, the branched model shows superior performance in all categories, once again demonstrating its ability to encode information within the branched signature.

\begin{figure}[h!]
    \centering
    \includegraphics[width=1\linewidth]{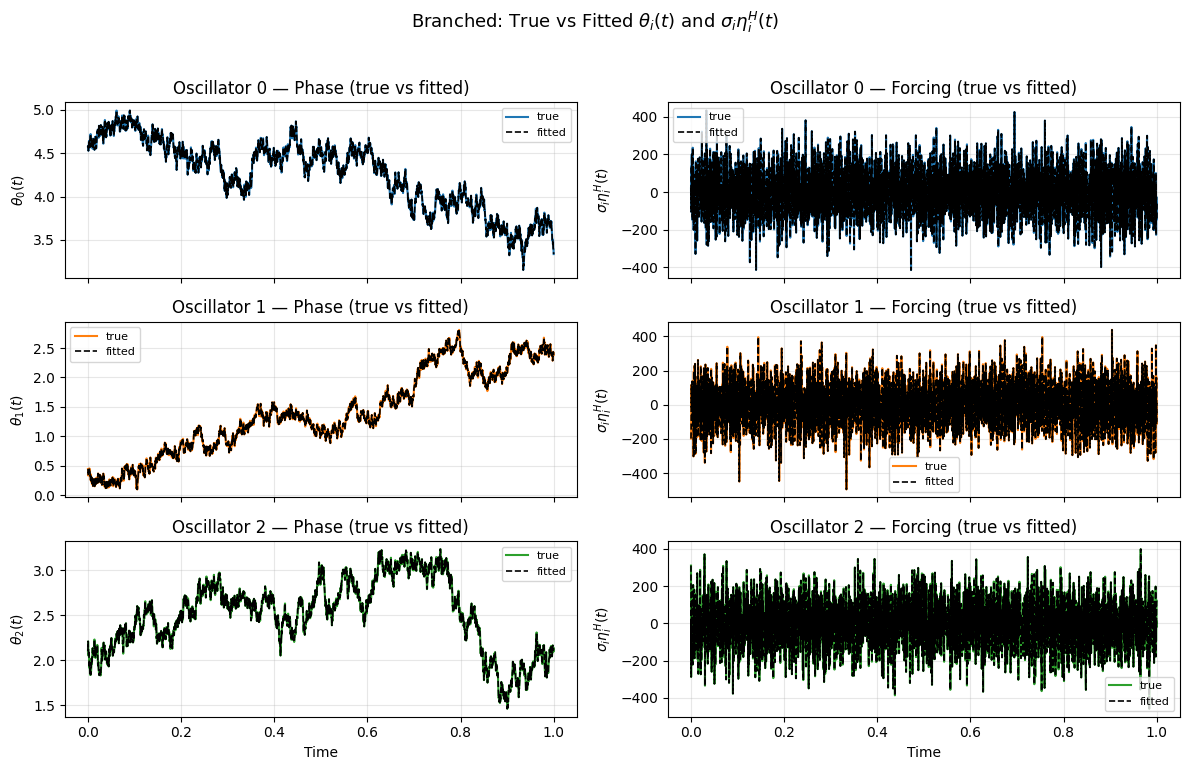}
    \caption{Solution and forcing fits for the branched signature kernel model on the noisy Kuramoto oscillator.}
    \label{fig:kuramoto}
\end{figure}

\begin{table}[ht]
    \centering
    \begin{tabular}{lccc}
        \hline
        Model & Forcing & \(d\theta / dt\) & \(\theta\) \\
        \hline
        Non-branched Signature Kernel Model & $1.73\mathrm{e}{-6}$  & $1.73\mathrm{e}{-6}$  & $1.78\mathrm{e}{-7}$  \\
        Branched Signature Kernel Model     & $\mathbf{5.25\mathrm{e}{-15}}$ & $\mathbf{1.31\mathrm{e}{-11}}$ & $\mathbf{1.33\mathrm{e}{-8}}$ \\
        \hline
    \end{tabular}
    \caption{Relative MSE of models for the Noisy Kuramoto Oscillator problem.}
    \label{tab:relative_mse_kuramoto}
\end{table}

%%%%%%%%%%%%%%%%%%%%%%%%%%%%%%%%%%
\section*{Acknowledgements}

% We would like to thank Mert G\"{u}rb\"{u}zbalaban, Yuanhan Hu and Yingli Wang. 
% Qi Feng is partially supported by the grants NSF DMS-2306769 and DMS-2420029. 
% Xiaoyu Wang is supported by the Guangzhou-HKUST(GZ) Joint Fund-
% ing Program (No.2025A03J3556) and Guangzhou Municipal Key Laboratory of Financial
% Technology Cutting-Edge Research (No.2024A03J0630).
% Lingjiong Zhu is partially supported by the grants NSF DMS-2053454 and DMS-2208303. 

\newpage
\bibliographystyle{apalike}
\bibliography{references}

\begin{thebibliography}{}

\bibitem[Abbasbandy et~al., 2015]{rkhs_nonlinear_bvps}
Abbasbandy, S., Azarnavid, B., and Alhuthali, M.~S. (2015).
\newblock A shooting reproducing kernel hilbert space method for multiple
  solutions of nonlinear boundary value problems.
\newblock {\em Journal of Computational and Applied Mathematics}, 279:293--305.

\bibitem[Ali and Feng, 2025]{ali2025branchedsignaturemodel}
Ali, M. and Feng, Q. (2025).
\newblock Branched signature model.

\bibitem[Alòs et~al., 2025]{alos2025volatilitymodelingroughpaths}
Alòs, E., Òscar Burés, de~Santiago, R., and Vives, J. (2025).
\newblock Volatility modeling with rough paths: A signature-based alternative
  to classical expansions.

\bibitem[Arias, 1970]{Arias1970measure}
Arias, A. (1970).
\newblock A measure of earthquake intensity.
\newblock {\em Seismic Design for Nuclear Power Plants, R.J.\ Hansen (ed.), MIT
  Press}, pages 438--483.

\bibitem[Bayer et~al., 2025]{bayer2025localregressionpathspaces}
Bayer, C., Gogolashvili, D., and Pelizzari, L. (2025).
\newblock Local regression on path spaces with signature metrics.

\bibitem[Bayraktar et~al., 2024]{bayraktar2024deep}
Bayraktar, E., Feng, Q., and Zhang, Z. (2024).
\newblock Deep signature algorithm for multidimensional path-dependent options.
\newblock {\em SIAM Journal on Financial Mathematics}, 15(1):194--214.

\bibitem[Boedihardjo et~al., 2016]{boedihardjo2016signature}
Boedihardjo, H., Geng, X., Lyons, T., and Yang, D. (2016).
\newblock The signature of a rough path: Uniqueness.
\newblock {\em Advances in Mathematics}, 293:720--737.

\bibitem[Caba{\~n}as et~al., 1997]{cabanas1997Ariasdestruction}
Caba{\~n}as, L., Benito, B., and Herr{\'a}iz, M. (1997).
\newblock An approach to the measurement of the potential structural damage of
  earthquake ground motions.
\newblock {\em Earthquake Engineering \& Structural Dynamics}, 26(1):79--92.

\bibitem[Ceylan et~al., 2026]{ceylan2026universal}
Ceylan, M., Kwossek, A.~P., and Pr{\"o}mel, D.~J. (2026).
\newblock Universal approximation with signatures of non-geometric rough paths.
\newblock {\em arXiv preprint arXiv:2602.05898}.

\bibitem[Chen et~al., 2021]{nonlinearpdes_GP_RKHS}
Chen, Y., Hosseini, B., Owhadi, H., and Stuart, A.~M. (2021).
\newblock Solving and learning nonlinear pdes with gaussian processes.

\bibitem[Chevyrev and Kormilitzin, 2025]{chevyrev2025primer}
Chevyrev, I. and Kormilitzin, A. (2025).
\newblock A primer on the signature method in machine learning.
\newblock In {\em Signature Methods in Finance: An Introduction with
  Computational Applications}, pages 3--64. Springer.

\bibitem[Chevyrev and Oberhauser,
  2022]{chevyrev2022signaturemomentscharacterizelaws}
Chevyrev, I. and Oberhauser, H. (2022).
\newblock Signature moments to characterize laws of stochastic processes.

\bibitem[Cohen et~al., 2023]{cohen2023nowcastingsignaturemethods}
Cohen, S.~N., Lui, S., Malpass, W., Mantoan, G., Nesheim, L., Áureo~de Paula,
  Reeves, A., Scott, C., Small, E., and Yang, L. (2023).
\newblock Nowcasting with signature methods.

\bibitem[Connes and Kreimer, 1999]{connes1999hopf}
Connes, A. and Kreimer, D. (1999).
\newblock Hopf algebras, renormalization and noncommutative geometry.
\newblock In {\em Quantum field theory: perspective and prospective}, pages
  59--109. Springer.

\bibitem[Cox et~al., 2026]{cox2026universal}
Cox, S., Khedher, A., and Maessen, T. (2026).
\newblock Universal approximation by signatures for infinite-dimensional rough
  paths.
\newblock {\em arXiv preprint arXiv:2603.03058}.

\bibitem[Cuchiero et~al., 2023]{cuchiero2023signature}
Cuchiero, C., Gazzani, G., and Svaluto-Ferro, S. (2023).
\newblock Signature-based models: theory and calibration.
\newblock {\em SIAM journal on financial mathematics}, 14(3):910--957.

\bibitem[Cuchiero et~al., 2025]{cuchiero2025universal}
Cuchiero, C., Primavera, F., and Svaluto-Ferro, S. (2025).
\newblock Universal approximation theorems for continuous functions of
  c{\`a}dl{\`a}g paths and l{\'e}vy-type signature models.
\newblock {\em Finance and Stochastics}, 29(2):289--342.

\bibitem[Deng et~al., 2012]{signal_processing_duffing}
Deng, X., Liu, H., and Long, T. (2012).
\newblock A new complex duffing oscillator used in complex signal detection.
\newblock {\em Chinese Science Bulletin}, 57.

\bibitem[Fang et~al., 2023]{ppde_nn_2}
Fang, B., Ni, H., and Wu, Y. (2023).
\newblock A neural rde-based model for solving path-dependent pdes.

\bibitem[{Federal Reserve Bank of St. Louis}, 2024]{fred_gdp}
{Federal Reserve Bank of St. Louis} (2024).
\newblock {Real Gross Domestic Product}.
\newblock https://fred.stlouisfed.org/series/A191RL1Q225SBEA.
\newblock FRED, Federal Reserve Bank of St. Louis. Accessed: 2026-05-19.

\bibitem[Feng et~al., 2023]{feng2023}
Feng, Q., Luo, M., and Zhang, Z. (2023).
\newblock Deep signature fbsde algorithm.
\newblock {\em Numerical Algebra, Control and Optimization}, 13(3\&4):500--522.

\bibitem[Fornberg and Flyer, 2015]{RBF_kernel_methods}
Fornberg, B. and Flyer, N. (2015).
\newblock Solving pdes with radial basis functions.
\newblock {\em Acta Numerica}, 24:215–258.

\bibitem[Friz and Victoir, 2010]{Friz_Victoir_2010}
Friz, P.~K. and Victoir, N.~B. (2010).
\newblock {\em Multidimensional Stochastic Processes as Rough Paths: Theory and
  Applications}.
\newblock Cambridge Studies in Advanced Mathematics. Cambridge University
  Press.

\bibitem[Genet and Inzirillo, 2025]{keras_sig}
Genet, R. and Inzirillo, H. (2025).
\newblock Keras sig: Efficient path signature computation on gpu in keras 3.

\bibitem[Green et~al., 2012]{magnet_duffing_noisy}
Green, P., Worden, K., Atallah, K., and Sims, N. (2012).
\newblock {\em The Benefits of Duffing-type Nonlinearities and Electrical
  Optimisation of a Randomly Excited Energy Harvester}, volume~6, pages
  657--667.

\bibitem[Gubinelli, 2010]{gubinelli2010ramification}
Gubinelli, M. (2010).
\newblock Ramification of rough paths.
\newblock {\em Journal of Differential Equations}, 248(4):693--721.

\bibitem[Gyurk{\'o} et~al., 2013]{gyurko2013extracting}
Gyurk{\'o}, L.~G., Lyons, T., Kontkowski, M., and Field, J. (2013).
\newblock Extracting information from the signature of a financial data stream.
\newblock {\em arXiv preprint arXiv:1307.7244}.

\bibitem[Hairer and Kelly, 2015]{hairer2015geometric}
Hairer, M. and Kelly, D. (2015).
\newblock Geometric versus non-geometric rough paths.
\newblock In {\em Annales de l'IHP Probabilit{\'e}s et statistiques},
  volume~51, pages 207--251.

\bibitem[Hambly and Lyons, 2010]{hambly2010uniqueness}
Hambly, B. and Lyons, T. (2010).
\newblock Uniqueness for the signature of a path of bounded variation and the
  reduced path group.
\newblock {\em Annals of Mathematics}, pages 109--167.

\bibitem[Hao et~al., 2025]{PINNproblem3}
Hao, B., Braga-Neto, U., Liu, C., Wang, L., and Zhong, M. (2025).
\newblock Stability in training pinns for stiff pdes: Why initial conditions
  matter.

\bibitem[Horvath et~al.,
  2023]{horvath2023optimalstoppingdistributionregression}
Horvath, B., Lemercier, M., Liu, C., Lyons, T., and Salvi, C. (2023).
\newblock Optimal stopping via distribution regression: a higher rank signature
  approach.

\bibitem[Issa et~al., 2023]{PPDE_nn_kernel}
Issa, Z., Horvath, B., Lemercier, M., and Salvi, C. (2023).
\newblock Non-adversarial training of neural sdes with signature kernel scores.

\bibitem[Kansa, 1990a]{kansa2}
Kansa, E. (1990a).
\newblock Multiquadrics – a scattered data approximation scheme with
  applications to computational fluid-dynamics. ii: Solutions to parabolic,
  hyperbolic and elliptic partial differential equations.
\newblock {\em Computers \& Mathematics with Applications}, 19:147--161.

\bibitem[Kansa, 1990b]{kansa1}
Kansa, E. (1990b).
\newblock Multiquadrics—a scattered data approximation scheme with
  applications to computational fluid-dynamics—i surface approximations and
  partial derivative estimates.
\newblock {\em Computers \& Mathematics with Applications}, 19(8):127--145.

\bibitem[Kidger et~al., 2019]{DeepSignatureTransforms}
Kidger, P., Bonnier, P., Perez~Arribas, I., Salvi, C., and Lyons, T. (2019).
\newblock Deep signature transforms.
\newblock In {\em Advances in Neural Information Processing Systems},
  volume~32. Curran Associates, Inc.

\bibitem[Kidger and Lyons, 2021]{kidger2021signatory}
Kidger, P. and Lyons, T. (2021).
\newblock {S}ignatory: differentiable computations of the signature and
  logsignature transforms, on both {CPU} and {GPU}.
\newblock In {\em International Conference on Learning Representations}.
\newblock \url{https://github.com/patrick-kidger/signatory}.

\bibitem[Kir{\'a}ly and Oberhauser, 2019]{kiraly2019kernels}
Kir{\'a}ly, F.~J. and Oberhauser, H. (2019).
\newblock Kernels for sequentially ordered data.
\newblock {\em Journal of Machine Learning Research}, 20(31):1--45.

\bibitem[Krishnapriyan et~al., 2021]{PINNproblems2}
Krishnapriyan, A., Gholami, A., Zhe, S., Kirby, R., and Mahoney, M. (2021).
\newblock Characterizing possible failure modes in physics-informed neural
  networks.
\newblock In Ranzato, M., Beygelzimer, A., Dauphin, Y., Liang, P., and Vaughan,
  J.~W., editors, {\em Advances in Neural Information Processing Systems},
  volume~34, pages 26548--26560. Curran Associates, Inc.

\bibitem[Kuramoto, 1975]{10.1007/BFb0013365}
Kuramoto, Y. (1975).
\newblock Self-entrainment of a population of coupled non-linear oscillators.
\newblock In Araki, H., editor, {\em International Symposium on Mathematical
  Problems in Theoretical Physics}, pages 420--422, Berlin, Heidelberg.
  Springer Berlin Heidelberg.

\bibitem[Lai and Leng, 2015]{stochastic_duffing_detection}
Lai, Z.-H. and Leng, Y.-G. (2015).
\newblock Generalized parameter-adjusted stochastic resonance of duffing
  oscillator and its application to weak-signal detection.
\newblock {\em Sensors}, 15(9):21327--21349.

\bibitem[Lemercier et~al.,
  2021]{lemercier2021distributionregressionsequentialdata}
Lemercier, M., Salvi, C., Damoulas, T., Bonilla, E.~V., and Lyons, T. (2021).
\newblock Distribution regression for sequential data.

\bibitem[Levin et~al., 2016]{levin2016learningpastpredictingstatistics}
Levin, D., Lyons, T., and Ni, H. (2016).
\newblock Learning from the past, predicting the statistics for the future,
  learning an evolving system.

\bibitem[Li, 2001]{variablecoefficient3}
Li, Q. (2001).
\newblock Free vibration of sdof systems with arbitrary time-varying
  coefficients.
\newblock {\em International Journal of Mechanical Sciences}, 43(3):759--770.

\bibitem[Li et~al., 2000]{variablecoefficient1}
Li, Q., Fang, J., and Liu, D. (2000).
\newblock Exact solutions for free vibration of single-degree-of-freedom
  systems with nonperiodically varying parameters.
\newblock {\em Journal of Vibration and Control}, 6:449--462.

\bibitem[Li, 1999]{variablecoefficient2}
Li, Q.~S. (1999).
\newblock A new exact approach for analyzing free vibration of sdof systems
  with nonperiodically time varying parameters.
\newblock {\em Journal of Vibration and Acoustics}, 122(2):175--179.

\bibitem[Lobo et~al., 2019]{stoch_duffing}
Lobo, D., Ritto, T., Castello, D., and Cataldo, E. (2019).
\newblock Dynamics of a duffing oscillator with the stiffness modeled as a
  stochastic process.
\newblock {\em International Journal of Non-Linear Mechanics}, 116:273--280.

\bibitem[Lyons, 2014]{lyons2014roughpathssignaturesmodelling}
Lyons, T. (2014).
\newblock Rough paths, signatures and the modelling of functions on streams.

\bibitem[Lyons and McLeod, 2025]{lyons2025signaturemethodsmachinelearning}
Lyons, T. and McLeod, A.~D. (2025).
\newblock Signature methods in machine learning.

\bibitem[Lyons et~al., 2020]{lyons2020non}
Lyons, T., Nejad, S., and Perez~Arribas, I. (2020).
\newblock Non-parametric pricing and hedging of exotic derivatives.
\newblock {\em Applied Mathematical Finance}, 27(6):457--494.

\bibitem[Lyons et~al., 2007]{lyons2007differential}
Lyons, T.~J., Caruana, M., and L{\'e}vy, T. (2007).
\newblock {\em Differential equations driven by rough paths: Ecole d'Et{\'e} de
  Probabilit{\'e}s de Saint-Flour XXXIV-2004}.
\newblock Springer.

\bibitem[Mann and Sims, 2009]{magnet_duffing}
Mann, B. and Sims, N. (2009).
\newblock Energy harvesting from the nonlinear oscillations of magnetic
  levitation.
\newblock {\em Journal of Sound and Vibration}, 319:515--530.

\bibitem[Mohaddes et~al.,
  2025]{mohaddes2025regularizedlearningfractionalbrownian}
Mohaddes, A., Iafrate, F., and Lederer, J. (2025).
\newblock Regularized learning for fractional brownian motion via path
  signatures.

\bibitem[Pannier and Salvi, 2024]{pannier2024pathdependentpdesolverbased}
Pannier, A. and Salvi, C. (2024).
\newblock A path-dependent pde solver based on signature kernels.

\bibitem[Ramtani, 2013]{ramtani2013contmechanics}
Ramtani, S. (2013).
\newblock Basic concepts and models in continuum damage mechanics.
\newblock Technical report, HAL Open Archive.
\newblock hal-00776729, \url{https://hal.science/hal-00776729/document}.

\bibitem[Reizenstein and Graham, 2020]{iisignature}
Reizenstein, J. and Graham, B. (2020).
\newblock Algorithm 1004: The iisignature library: Efficient calculation of
  iterated-integral signatures and log signatures.
\newblock {\em ACM Transactions on Mathematical Software (TOMS)}.

\bibitem[Reynolds et~al., 2014]{timberduffingmodel}
Reynolds, T., Harris, R., and Chang, W.-S. (2014).
\newblock Nonlinear pre-yield modal properties of timber structures with
  large-diameter steel dowel connections.
\newblock {\em Engineering Structures}, 76:235--244.

\bibitem[Sabate-Vidales et~al., 2020]{ppde_nn_1}
Sabate-Vidales, M., Šiška, D., and Szpruch, L. (2020).
\newblock Solving path dependent pdes with lstm networks and path signatures.

\bibitem[Sakaguchi, 1988]{sakaguchi1988cooperative}
Sakaguchi, H. (1988).
\newblock Cooperative phenomena in coupled oscillator systems under external
  fields.
\newblock {\em Progress of theoretical physics}, 79(1):39--46.

\bibitem[Salvi et~al., 2021]{Salvi_2021_goursat}
Salvi, C., Cass, T., Foster, J., Lyons, T., and Yang, W. (2021).
\newblock The signature kernel is the solution of a goursat pde.
\newblock {\em SIAM Journal on Mathematics of Data Science}, 3(3):873–899.

\bibitem[Shmelev and Salvi, 2025]{shmelev2025pysiglibfastsignaturebased}
Shmelev, D. and Salvi, C. (2025).
\newblock pysiglib -- fast signature-based computations on cpu and gpu.

\bibitem[Solow, 1956]{SolowModel}
Solow, R.~M. (1956).
\newblock A contribution to the theory of economic growth.
\newblock {\em The Quarterly Journal of Economics}, 70(1):65--94.

\bibitem[T{\'o}th et~al., 2025]{ksiglibrary}
T{\'o}th, C., Cruz, D. J.~D., and Oberhauser, H. (2025).
\newblock A user's guide to \texttt{KSig}: Gpu-accelerated computation of the
  signature kernel.
\newblock {\em arXiv preprint arXiv:2501.07145}.

\bibitem[Toth et~al., 2024]{toth2024randomfouriersignaturefeatures}
Toth, C., Oberhauser, H., and Szabo, Z. (2024).
\newblock Random fourier signature features.

\bibitem[{Vibrationdata}, ]{elcentroearthquakedata}
{Vibrationdata}.
\newblock El centro earthquake.

\bibitem[Wang et~al., 2024]{PINNproblem4}
Wang, S., Sankaran, S., and Perdikaris, P. (2024).
\newblock Respecting causality for training physics-informed neural networks.
\newblock {\em Computer Methods in Applied Mechanics and Engineering},
  421:116813.

\bibitem[Wang et~al., 2021]{PINNproblems1}
Wang, S., Teng, Y., and Perdikaris, P. (2021).
\newblock Understanding and mitigating gradient flow pathologies in
  physics-informed neural networks.
\newblock {\em SIAM Journal on Scientific Computing}, 43(5):A3055--A3081.

\bibitem[Xiang et~al., 2024]{SR_duffing}
Xiang, J., Guo, J., and Li, X. (2024).
\newblock A two-stage duffing equation-based oscillator and stochastic
  resonance for mechanical fault diagnosis.
\newblock {\em Chaos, Solitons \& Fractals}, 182:114755.

\bibitem[Yilmaz and Unal, 2019]{duffing_FGN_modeling}
Yilmaz, A. and Unal, G. (2019).
\newblock Stochastic duffing equation in modelling of financial time series.
\newblock {\em International Journal of Dynamics and Control}, 7:1--22.

\bibitem[Zeng and Jiang,
  2025]{zeng2025semiparametricfunctionalclassificationpath}
Zeng, P. and Jiang, S. (2025).
\newblock Semi-parametric functional classification via path signatures
  logistic regression.

\end{thebibliography}

\end{document}